\documentclass{article}
\usepackage[utf8]{inputenc}

\usepackage{graphicx,amsfonts,mathrsfs,todonotes,xcolor}

\usepackage{authblk}

\usepackage{tikz}
\usetikzlibrary{arrows,calc}

\tikzstyle{arc}=[->, shorten <=3pt, shorten >=3pt, >=stealth, line width=1.5pt]
\tikzstyle{edge}=[shorten <=2pt, shorten >=2pt, >=stealth, line width=1.5pt]
\tikzstyle{vertex}=[circle, fill=white, draw, minimum size=6pt, inner sep=0pt,
                    outer sep=0pt]

\usepackage{float}

\usepackage{xcolor}

\usepackage{amsmath}
\usepackage{amssymb,hyperref}
\usepackage{bm} 

\usepackage{amsthm,geometry,cite}

\usepackage[capitalize]{cleveref}

\providecommand{\keywords}[1]
{
    \vspace{0.5em}
  \small	
  \textbf{Key words.} #1
}

\providecommand{\AMS}[1]
{
    \vspace{1em}
  \small	
  \textbf{AMS subject classifications.} #1
}

\newtheorem{theorem}{Theorem}
\newtheorem{definition}{Definition}
\newtheorem{conjecture}{Conjecture}
\newtheorem{lemma}[theorem]{Lemma}
\newtheorem{corollary}[theorem]{Corollary}
\newtheorem{proposition}[theorem]{Proposition}

\newtheorem{observation}[theorem]{Observation}
\newtheorem{example}[theorem]{Example}

\newcommand{\red}[1]{#1}

\DeclareMathOperator{\Aut}{Aut}
\DeclareMathOperator{\minority}{minority}
\DeclareMathOperator{\majority}{majority}

\DeclareMathOperator{\NP}{NP}

\DeclareMathOperator{\Csp}{CSP}
\DeclareMathOperator{\Sym}{Sym}

\DeclareMathOperator{\sw}{sw}
\newcommand{\bA}{{\mathfrak A}}
\newcommand{\bB}{{\mathfrak B}}
\newcommand{\bC}{{\mathfrak C}}
\newcommand{\bR}{{\mathfrak R}}
\newcommand{\bH}{{\mathfrak H}}
\newcommand{\mi}{\boldsymbol{-}}

\title{Forbidden Tournaments \\
and the Orientation Completion Problem\thanks{Both authors have been funded by the European Research Council (Project POCOCOP, ERC Synergy Grant 101071674). Views and opinions expressed are however those of the authors only and do not necessarily reflect those of the European Union or the European Research Council Executive Agency. Neither the European Union nor the granting authority can be held responsible for them.}}

\author[1]{Manuel Bodirsky\thanks{manuel.bodirsky@tu-dresden.de}}
\author[1]{Santiago Guzm\'an-Pro\thanks{santiago.guzman\_pro@tu-dresden.de}}

\affil[1]{Institut f\"ur Algebra, TU Dresden}

\begin{document}
\date{}

\maketitle

\begin{abstract}
   For a fixed finite set of finite tournaments ${\mathcal F}$, the \emph{${\mathcal F}$-free orientation problem} asks whether a given finite undirected graph $G$ has an \emph{$\mathcal F$-free orientation}, i.e., whether the edges of $G$ can be oriented so that the resulting digraph does not embed any of the tournaments from ${\mathcal F}$. We prove that for every ${\mathcal F}$, this problem is in P or NP-complete. Our proof reduces the classification task to a complete complexity classification of the \emph{orientation completion problem} for ${\mathcal F}$, which is the variant of the problem above  where the input is a directed graph instead of an undirected graph, introduced by Bang-Jensen, Huang, and Zhu (2017). 
   Our proof uses results from the theory of constraint satisfaction, and a result of Agarwal and Kompatscher (2018) about infinite permutation groups and transformation monoids. 
\end{abstract}

\keywords{Orientations with forbidden patterns, homomorphisms, infinite graphs, computational complexity}

\AMS{05C20, 05C60, 05C63, 05C85, 03C98}



\section{Introduction}

 
For a fixed finite set of finite oriented graphs ${\mathcal F}$, the \emph{${\mathcal F}$-free
orientation problem} asks whether a given finite undirected graph $G$ has an \emph{$\mathcal F$-free
orientation}, i.e., whether the edges of $G$ can be oriented so that the resulting digraph does not
contain any $F\in \mathcal F$ as an induced oriented graph. This problem was first studied
from a structural perspective: Skrien~\cite{skrienJGT6} proposed structural characterisations of
graphs that admit an $\mathcal F$-free orientation, where $\mathcal F$ is a fixed set of oriented
paths on $3$-vertices. The most notorious result in this direction that a connected graph
$G$ is a proper circular-arc graph if and only if it admits a $\{(\{1,2,3\},\{(1,2),(1,3)\}),
(\{1,2,3\},\{(2,1),(3,1)\})\}$-free orientation~\cite{skrienJGT6}.

From an algorithmic perspective, 
the $\mathcal F$-free
orientation problem can be easily reduced to $2$-SAT if $\mathcal F$ is a set of oriented
paths on $3$ vertices~\cite{bangjensenJCTB59}. Later, in \cite{GuzmanProForbiddenPatterns}
the authors extended the previous observation, and showed that for several sets of oriented
graphs on $3$ vertices the $\mathcal F$-free orientation problem is in P --- leaving only
open the symmetric cases \begin{align*}
\mathcal F & = \{(\{1,2,3\},\{(1,2),(1,3)\}), (\{1,2,3\},\{(1,2),(2,3),
(3,1)\})\} \\
\text{ and } \mathcal F & = \{(\{1,2,3\},\{(2,1),(3,1)\}), (\{1,2,3\},\{(1,2),(2,3), (3,1)\})\}.
\end{align*}
On the other hand, the Roy-Gallai-Hasse-Vitaver Theorem~\cite{gallaiPCT,
hasseIMN28, vitaverDAN147, royIRO1} implies that the $k$-colouring problem can be
interpreted as an $\mathcal F$-free orientation problem, i.e., there is a finite set of 
oriented graphs $\mathcal F_k$ such that a graph $G$ is $k$-colourable if and only if it
admits an $\mathcal F_k$-free orientation. This yields natural instances of NP-hard 
$\mathcal F$-free orientation problems.

Recently, a thorough study of problems that stem from orientation problems
was initiated by Bang-Jensen, Huang, and Zhu~\cite{bangjensenJGT87}.
For a fixed class of oriented graphs $\mathcal{C}$, the \textit{orientation
completion problem} asks whether a partially oriented  graph $G$ can be completed
to an oriented graph in $\mathcal{C}$ by orienting its non-oriented edges. In
particular, given a finite set of oriented graphs $\mathcal F$, the \emph{$\mathcal F$-free
orientation completion problem} generalises the $\mathcal F$-free orientation problem.
In \cite{bangjensenJGT87}, the authors study the complexity of the orientation
completion problem for several classes of digraphs such as in-tournaments, local 
tournaments, and locally transitive tournaments.  They show that for each of these classes,
the orientation completion problem is in P or $\NP$-complete.

In this work, we prove that for any finite set of finite tournaments ${\mathcal F}$, the
$\mathcal F$-free orientation problem is in P or NP-complete. Our proof reduces the
classification task to the $\mathcal F$-free orientation
completion problem, for which we then provide a complete complexity classification. This
intermediate result is interesting in its own right,  addressing the study initiated
by Bang-Jensen, Huang, and Zhu~\cite{bangjensenJGT87}.

The rest of this work is structured as follows. In \cref{sect:MT}, we introduce
the necessary model theoretic background for this work. Similarly, in \cref{sect:CSP},
we provide context for the theory of constraint satisfaction. In \cref{sect:dicho-oc},
we prove that for each finite set of tournaments $\mathcal F$, there is a Boolean structure
whose CSP
is polynomial-time equivalent to the $\mathcal F$-free orientation problem. This yields
a dichotomy for the $\mathcal F$-free orientation completion problem, and moreover,
we propose a complete classification of the complexity of these problems in terms of
$\mathcal F$ and the $\mathcal F$-free tournaments. In \cref{sect:dicho-o},
we build on the previous classification to classify the complexity of the 
$\mathcal F$-free orientation problem. We will do so by using a result of Agarwal and
Kompatscher~\cite{AgarwalKompatscher} about infinite permutation groups and transformation monoids
which we introduce in \cref{sect:sym}.
Finally, in Section~\ref{sect:examples} we present our classification from \cref{sect:dicho-oc,sect:dicho-o} in graph theoretic terms,
and propose some applications which we believe are interesting in their own right to graph theorists. Moreover, these examples can help readers that are less familiar with constraint satisfaction techniques to gain intuition before reading the technical proofs of this paper --- readers with a model theory or constraint satisfaction background and motivation can skip this section.

In the remaining of this section we first extend the motivation of our work, by mentioning
some relations of the $\mathcal F$-free orientation  problem to previously studied problems
in graph theory, finite model theory, constraint satisfaction theory, and infinite model theory. We conclude this section by
introducing the formal setting under which we study the orientation and orientation completion problems. 

\subsection{Related Work}

Similar to $\mathcal F$-free orientation problems, Damaschke~\cite{damaschke1990} considers
characterisations of graph classes by means of finitely many forbidden ordered graphs. For
instance, if $(P_3,\le)$ is the linear ordering of $P_3 = (\{1,2,3\},\{12,13\})$ where $1\le 2\le 3$,
then a graph $G$ is chordal if and only if it admits a
$(P_3,\le)$-free 
linear ordering~\cite{damaschke1990}.  Shortly after, Duffus, Ginn, and R\"odl~\cite{duffusRSA7} 
consider the complexity classification of the following ordering problem: given a linearly ordered graph
$(G,\le)$, decide if an input graph $H$ admits a $(G,\le)$-free linear ordering. They showed that
for almost all $2$-connected graphs $G$, the  $(G,\le)$-free ordering is NP-complete, and
conjectured that this is the case for all $2$-connected graphs unless $G$ is a clique. Similar
problems have been considered for circular orderings~\cite{guzmanAMC438}, and for so-called
tree-layouts of graphs~\cite{paulPREPRINT} 
(for a comprehensive  study of such problems from a structural perspective see~\cite{guzmanPHD}).
Here we initiate a parallel study to the one started by Duffus, Ginn, and 
R\"odl~\cite{duffusRSA7} by considering  orientation instead of ordering problems.

Both the ${\mathcal F}$-free orientation problem and the ${\mathcal F}$-free orientation completion problem can be viewed as special cases of a significantly larger class of computational problems, namely the class of problems that can be expressed in the logic \emph{MMSNP$_2$}. In the context of digraphs, a computational problem expressed in 
MMSNP$_2$ asks ``given a digraph $D$,  is there an edge colouring of $D$ that avoids 
some fixed finite set of edge-coloured digraphs?''.
This logic relates to Feder and Vardi's famous logic \emph{MMSNP} (for \emph{monotone monadic strict NP}) as Courcelle's logic MSO$_2$ relates to MSO~\cite{BarsukovThesis}. It has the same expressive power as \emph{guarded disjunctive Datalog}~\cite{OBDA}, which is a formalism studied in database theory.
It can be shown that every CSP in MMSNP$_2$ can be expressed as a Constraint Satisfaction Problem (CSP) for a  reduct of a finitely bounded homogeneous structure~\cite{BodirskyKnaeuerStarke}, and hence falls into the scope of the so-called \emph{tractability conjecture}~\cite{BPP-projective-homomorphisms}. This conjecture
states that such CSPs 
are in P or NP-complete, and even provides a mathematical condition 
to describe the boundary between the cases in P and the NP-complete cases.
This condition has numerous equivalent characterisations~\cite{Book,BartoPinskerDichotomy,BKOPP-equations}, but despite recent progress~\cite{MottetPinskerSmooth} the tractability conjecture for MMSNP$_2$ is still wide open. 
In contrast, the P versus NP-complete dichotomy is true for MMSNP~\cite{FederVardi,Kun}
(using the complexity dichotomy for finite-domain CSPs~\cite{ZhukFVConjecture,BulatovFVConjecture}), and even the tractability conjecture has been verified in this case~\cite{MMSNP-Journal}.
The graph orientation problems studied here cannot be expressed in MMSNP, but it is straightforward to formulate them in MMSNP$_2$. Our result not only shows a complexity dichotomy for the ${\mathcal F}$-free orientation problems, but verifies the tractability conjecture for this subclass of MMSNP$_2$.



\subsection{Formal Setting}
\label{sect:setting} 

This work lies in the intersection of graph theory, model theory, and computational
complexity. For this reason, we begin by carefully introducing basic notation
and nomenclature; in this section, we start with terminology from graph theory.
A \textit{digraph} $D$ consists of a vertex set $V(D)$ and a
binary relation $E(D) \subseteq V(D)^2$. The elements of $E(D)$ are called edges,
and whenever there are vertices $x,y\in V(D)$ such that $(x,y) \in E(D)$ and $(y,x) \in E(D)$,
we write $xy\in E(D)$, and we say that $xy$ is a \textit{symmetric edge} of $D$ 
--- notice that in graph theory an edge might refer to a symmetric edge in this context. Most digraphs considered
here will be \emph{loopless}, i.e., they do not contain edges of the form $(x,x)$ for $x \in V(D)$.
Whenever the digraph is clear from the context, we will simply write $V$ for $V(D)$, and
$E$ for $E(D)$.

A (loopless) \textit{graph} $G$ will be viewed as a (loopless) digraph where 
$E(G)$ is a symmetric relation, and an \emph{oriented graph} $O$ is a digraph where
$E(O)$ is an anti-symmetric relation (i.e., none of the edges in $E(O)$ is symmetric).
It will be convenient to denote by $U$ the
relation $E\cup E^{-1}$, i.e, given a digraph $D$ we have
$(x,y) \in U$ if and only if $(x,y)\in E(D)$ or $(y,x)\in E(D)$.
The \textit{underlying graph} of a digraph $D$ is the graph $u(D)$ with vertex set
$V(D)$ and edge set $U(D) = U$. The relation $U$ will be very important and highly
used in this work so, when $D$ is clear from context
we will write $U$ for $U(D)$. Given a digraph $D$, we will write
$(V,U)$ for the underlying graph of $U$, and $(D,U)$ for the structure that contains
both the oriented edges and the edge relation of the underlying graph.

 An \textit{orientation} of a graph $G$ is an oriented graph
$G'$ such that $u(G') = G$. A \textit{tournament} is an orientation of the complete
graph $K_n$ with $n$ vertices, for some $n \geq 1$, and a graph is \textit{semicomplete} if its underlying graph is complete. We denote by $T_n$ the 
\textit{transitive tournament} on $n$ vertices, and by $\overrightarrow{C}_n$ the
directed cycle on $n$ vertices.

Given digraphs $D$ and $H$ we say that $D$ is a \textit{subdigraph} of $H$
if $V(D)\subseteq V(H)$ and $E(D)\subseteq E(H)$.
 We say that $D$ is 
\textit{spanning} subdigraph of $H$ if $V(D) = V(H)$. In particular,
if $G'$ is an orientation of a graph $G$, then $G'$ is a spanning subdigraph
of $G$.  
A \textit{homomorphism}
$\varphi\colon D\to H$ is a function $\varphi
\colon V(D)\to V(H)$ such that if $(x,y)\in E(D)$, then $(\varphi(x),
\varphi(y)) \in E(H)$ --- this notion naturally generalises to homomorphisms
of relational structures, which we introduce later. If such a homomorphism exists, 
we write $D\to H$, otherwise we write $D\not\to H$. An \textit{embedding} is 
an injective homomorphism $\varphi\colon D\to H$ such that $(x,y)\in E(D)$ if and
only if $(\varphi(x), \varphi(y)) \in E(H)$. Notice that if $H$ is an oriented graph
and $T$ is a tournament, then every homomorphism $\varphi\colon T\to H$ is an
embedding. Finally, given a set of digraphs $\mathcal F$, we say that a digraph
$D$ is $\mathcal F$\textit{-free} if there is no embedding $\varphi\colon F\to D$ for
any $F\in \mathcal F$. We now formalize the orientation and orientation completion
problems in the setting described above.

    {\vspace{1em} \noindent \scshape  $\mathcal F$-free orientation problem\par}%
    \begin{itemize}\vspace{-0.5em}
        \item  Input:  a finite graph $G$;
        \item Question: is there an $\mathcal F$-free  orientation $G'$ of $G$?
     \end{itemize}

     {\vspace{0.5em} \noindent \scshape  $\mathcal F$-free orientation completion problem\par}%
     
    \begin{itemize}\vspace{-0.5em}
        \item Input: a finite digraph $D$;
     
     \item  Question: is there a spanning subdigraph $D'$ of $D$ such that
     $D'$ is an $\mathcal F$-free oriented graph?
     \end{itemize}
    
Clearly, for every fixed set of oriented graphs  ${\mathcal F}$ both the ${\mathcal F}$-free orientation and the
${\mathcal F}$-free orientation completion problem are in the complexity class NP.
As previously mentioned, this work studies the $\mathcal F$-free orientation (completion)
problem when $\mathcal F$ is a set of tournaments.
Given a tournament $T$, we will often write $T$-free orientation (completion)
problem instead of $\{T\}$-free orientation problem. For instance, since every
graph admits an acyclic orientation, the $\overrightarrow{C_3}$-free orientation
problem is trivial and polynomial-time tractable. On the contrary, the $\overrightarrow{C_3}$-free
orientation completion problem is NP-complete (\cref{cor:3-vertices-or-com}).
Note that if ${\mathcal F}$ contains a tournament with only one vertex, then there is no 
${\mathcal F}$-free oriented graph with a non-empty set of vertices, and hence both of the 
computational problems above are trivial. So we tacitly assume from now on that all tournaments in 
${\mathcal F}$ have at least two vertices.

\section{Model Theory Set Up}
\label{sect:MT}
Relational structures generalise digraphs
and allow multiple relations with relations of arbitrary finite arity. They are a natural tool in our study of graph problems. 
To define them formally, we need the concept of a \emph{relational signature}, which is a set $\tau$ of relation symbols $R,S,\dots$, each equipped with an arity $k \in {\mathbb N}$. A $\tau$-structure $\bA$ consists of a set $A$ (the domain) and for each $R \in \tau$ of arity $k$ a relation 
$R^{\bA} \subseteq A^k$ (in general, for structures  $\bA,\bB,\bC\dots$ we will denote by $A,B,C,\dots$ their corresponding domains). 
Clearly, a digraph $D$ (and hence also a graph) can be viewed as a structure with domain $V(D)$ and the signature $\{E\}$ where $E$ is a binary relation symbol that 
denotes the edge relation $E(D)$ of the graph. 

A \emph{substructure} of $\tau$-structure $\bA$ is a $\tau$-structure $\bB$ such that 
for every $R \in \tau$ of arity $k$ we have $R^{\bB} = R^{\bA} \cap B^k$; however, note that a substructure of a digraph when viewed as a structure corresponds to \emph{induced} subgraphs in graph theory, rather than subgraphs as introduced earlier. 
The \emph{union} of two $\tau$-structures $\bA$ and $\bB$ is the $\tau$-structure $\bC$ with domain $C := A \cup B$ and the relation $R^{\bC} := R^{\bA} \cup R^{\bB}$ 
for every $R \in \tau$. 
If $\bA$ is a $\tau$-structure and $\bA'$ is a $\sigma$-structure with the same domain, for $\sigma \subseteq \tau$,
and $R^{\bA} = R^{\bA'}$ for every 
$R \in \sigma$, 
then $\bA'$ is called a \emph{reduct} of $\bA$,
and $\bA$ is called a \emph{expansion} of $\bA'$. 

A \emph{homomorphism} between two $\tau$-structures $\bA$ and $\bB$ is a function $h \colon A \to B$ such that for every $R \in \tau$ of arity $k$ we have $a = (a_1,\dots,a_k) \in R^{\bA} \Rightarrow f(a) := (f(a_1),\dots,f(a_k)) \in R^{\bB}$. 
The \emph{constraint satisfaction problem} of a $\tau$-structure $\bB$ is the class of all finite $\tau$-structures $\bA$ with a homomorphism to $\bB$; for fixed $\bB$ of finite relational signature $\tau$, this class can be viewed as a computational problem, where the input consists of an arbitrary finite $\tau$-structure $\bA$, and the question is to decide whether $\bA \in \Csp(\bB)$. For example, $\Csp(K_3)$ can be viewed as the famous graph 3-colouring problem. 
Two structures are called \emph{homomorphically equivalent} if there are \red{homomorphisms} between the structures in both ways. Clearly, homomorphically equivalent structures have the same CSP. 

An \emph{endomorphism of $\bA$} is a homomorphism from $\bA$ to $\bA$. The set of all endomorphisms of a structure $\bA$ forms a transformation monoid.
It is well-known that a transformation monoid $M$ is an endomorphism monoid of a relational structure if and only if it is \emph{locally closed}, i.e., if $f \colon A \to A$ is such that for every $n \in {\mathbb N}$ and $a \in A^n$ there exists $g \in M$ such that $f(a) = g(a)$, then $\red{f} \in M$;
 these are the closed sets of the topology of \emph{pointwise convergence}, which is the product topology on $A^A$ where the \red{topology} on $A$ is taken to be discrete. 

A homomorphism is called \emph{strong} if the implication $\Rightarrow$ in the definition of homomorphisms is replaced by an equivalence $\Leftrightarrow$. 
An \emph{embedding} is an injective strong homomorphism.  A structure $\bA$ is called a \emph{core} if all endomorphisms of $\bA$ are embeddings. 
An \emph{isomorphism} is a bijective embedding. An \emph{automorphism}
of a structure $\bA$ is an isomorphism of $\bA$ with itself. Note that if $\bA$ is a finite core structure, then all endomorphisms of $\bA$ are automorphisms; this statement is false for general infinite structures $\bA$. 

A structure is called \emph{homogeneous} if every isomorphism between finite substructrues can be extended to an automorphism. 
It is a well-known fact that a permutation group is the automorphism group of a relational structure $\bA$ if and only if it is closed with respect to the restriction of the topology on $A^A$ above to the set of all permutations, which we denote by $\Sym(A)$. 
If $P$ is a set of permutations, we write $\langle P \rangle$ for the smallest permutation group that contains $P$ and is closed in $\Sym(A)$. 
\red{Homogeneous structures with finite relational signature are known to be \emph{$\omega$-categorical}, i.e., all countable model of their first-order theory are isomorphic~\cite{Hodges}.}

In our study of graph orientation problems, it will be convenient to pass from classes of finite structures to a single countably infinite structure, using Fra\"iss\'e theory. We do not need the full power of the theory, but exclusively work with the following fact. 

\begin{theorem}[see, e.g.,~\cite{Hodges}]
\label{thm:Fraisse}
Let $\tau$ be a finite relational signature. 
Let $\mathcal C$ be a class of finite $\tau$-structures which is closed under
substructures, isomorphisms, and unions.
Then there exists a countably infinite homogeneous $\tau$-structure $\bA$ such that $\mathcal C$ equals the class of finite $\tau$-structures that have an embedding into $\bA$. The structure $\bA$ is unique up to isomorphism, and called the \emph{Fra{\"i}ss\'e-limit} of ${\mathcal C}$.  
\end{theorem}

\begin{example}\label{expl:Rado}
    Let $\mathcal C$ be the class of all 
    finite graphs. 
    Then ${\mathcal C}$ satisfies the assumptions of Theorem~\ref{thm:Fraisse}, and the Fra{\"i}ss\'e-limit of $\mathcal C$ is called the \emph{Rado graph}, which we denote by $\bR$. 
    For $n \geq 2$, the \emph{Henson graph $\bH_n$} is the Fra{\"i}ss\'e-limit
    of the class $\mathcal C$ of all 
    finite graphs that do not embed $K_n$. 
\end{example}

A class $\mathcal C$ of finite $\tau$-structures is called \emph{finitely bounded} if there exists a finite set ${\mathcal F}$ of finite $\tau$-structures such that $\bA \in {\mathcal C}$ if and only if there is no structure in ${\mathcal F}$ which embeds into $\bA$; then \red{we} refer to the elements of ${\mathcal F}$ as the \emph{bounds} for ${\mathcal C}$.
We say that a structure $\bB$ is \red{\emph{finitely bounded}} if the class of all finite structures $\bA$ that embeds into $\bB$ is finitely bounded. 

\begin{example}\label{expl:Henson}
    The Rado graph $\bR$ and the Henson graphs are finitely bounded. For the Rado graph, the bounds contain the one-vertex loop graph, and the two-vertex graph with a single \red{directed} edge (which forces the edge relation of the Fra{\"i}ss\'e-limit to be loopless and symmetric). 
\end{example}

We now turn to definitions that are specific for our study of graph orientation problems.
Let $\mathcal F$ be a finite set of finite tournaments. Then the class of all finite 
${\mathcal F}$-free oriented graphs satisfies the conditions from Theorem~\ref{thm:Fraisse}
(here we use our general assumption that all tournaments in ${\mathcal F}$ have at least two vertices), 
and hence has a countably infinite homogeneous Fraiss\'e-limit $D_{\mathcal F} = (V;E)$. 
Note that $D_{\mathcal F}$ is finitely bounded: as bounds we take the structures in ${\mathcal F}$ and
additionally the loop digraph and the two-element digraph which contains a symmetric edge. Clearly,
$\Csp(D_{\mathcal F})$ is in P since it suffices to check whether \red{the} given directed graph contains
one of the tournaments from ${\mathcal F}$ as a subgraph, which can be tested in polynomial time (where
the degree of the polynomial is bounded by the maximal number of elements of the structures in ${\mathcal F}$).

The infinite structures $D_\mathcal F$ turn out to be closely related to $\mathcal F$-free
orientation (completion) problems via the following
expansions and reducts.
    
\begin{itemize}
    \item Let $H_{\mathcal F} = (V;U)$ be the underlying graph of 
    $D_{\mathcal F}$, that is, $U =
    E \cup E^{-1}$. Then a finite undirected graph $G$ has an ${\mathcal F}$-free
    orientation if and only if it has a homomorphism to $H_{\mathcal F}$. 
    Conversely, a finite directed graph has a homomorphism to $H_{\mathcal F}$ if and only if the underlying graph has an ${\mathcal F}$-free orientation.
    Hence,
    the ${\mathcal F}$-free orientation problem and $\Csp(H_{\mathcal F})$ are essentially the same problem. 
    \item Similarly as in the previous item, the ${\mathcal F}$-free orientation completion problem may be viewed as
    $\Csp(D_{\mathcal F},U)$. 
\end{itemize}

If all tournaments in ${\mathcal F}$ contain a directed cycle, then every finite graph $G$ has an ${\mathcal F}$-free orientation, since we may orient the edges along an arbitrary linear order of the vertices. Hence, the $\mathcal F$-free orientation problem is trivial and in P. Note that in this case, $H_{\mathcal F}$ is the Rado graph. 

Otherwise, there exists a smallest $n = n_{\mathcal F} \in {\mathbb N}$ 
such that $\mathcal F$ contains the transitive tournament $T_n$ with $n$ vertices. In this case,
there exists a largest $k = k_{\mathcal F} \in {\mathbb N}$ such that $K_k$ has an ${\mathcal F}$-free orientation~\cite{erdosMTA9}.
Finally, $m_{\mathcal F}$ denotes the \red{maximum number of vertices of a tournament in} ${\mathcal F}$. 
Now we present some easy observations.

\begin{observation}\label{obs:kandn}
It follows from the definition of $n_\mathcal F$ that if  $k \le n_\mathcal F -1$, then $T_k$ is $\mathcal F$-free. Hence, the inequality
$k_{\mathcal F} \ge  n_{\mathcal F} - 1$ holds. 
Since every tournament in $\mathcal F$ has at least two vertices, we have $n_\mathcal F \ge 2$ and $k_\mathcal F \ge 1$.
\end{observation}

\begin{observation}
A notable special case is the situation that the inequality from the previous observation is actually an equality. 
\begin{itemize}
\item If $n_{\mathcal F} = 2$ then no graph with a non-empty edge set has an ${\mathcal F}$-free orientation, $k_{\mathcal F} = 1$, and the ${\mathcal F}$-free orientation problem is trivial and in P. 
\item Suppose that there is no $\mathcal F$-free tournament with $n_{\mathcal F}$ vertices. 
    Then the ${\mathcal F}$-free orientation problem is equivalent to finding an $n_\mathcal F$-element clique in a
    given finite graph, and hence in P. This generalises the previous two cases. Note that in this case we have
    $k_{\mathcal F} = n_{\mathcal F} -1$. 
\end{itemize}
\end{observation}


\begin{lemma}\label{lem:HF-core}
    Let $\mathcal F$ be a finite set of finite tournaments. If $n_{\mathcal F} \geq 3$,
    then the structure $H_{\mathcal F}$ is a  core. 
\end{lemma}
\begin{proof}
First consider the case that $k := k_{\mathcal F}$ as defined above is even.  
    The relation $x \neq y$ can be defined primitively positively by the formula
    \begin{align*}
        \exists u_1,\dots,u_k \big (U(x,u_1) \wedge \cdots \wedge U(x,u_{k/2}) \wedge U(u_{k/2+1},y) \wedge \cdots \wedge U(u_k,y) \big). 
    \end{align*}
    Note that if $x=y$ then $x,u_1,\dots,u_k$ would induce a clique which does not have an ${\mathcal F}$-free orientation. 
    Similarly, one may show that 
 the complement of the edge relation has a primitive positive definition by removing \red{an} edge from $K_{k+1}$. This shows that every endomorphism of $H_{\mathcal F}$ is an embedding.
 The case that $k$ is odd can be handled similarly. 
\end{proof}

\section{Constraint Satisfaction Preliminaries}
\label{sect:CSP}
\red{Finite-domain CSPs exhibit a complexity dichotomy~\cite{Zhuk20}; this result has been announced in 2017
by Bulatov~\cite{BulatovFVConjecture} and by Zhuk~\cite{ZhukFVConjecture}, confirming a conjecture of Feder and Vardi~\cite{FederVardi}.} To state their result in its strongest form, we need to introduce the concept of a \emph{polymorphism} of a $\tau$-structure $\bB$, which is a \red{homomorphism} $f \colon B^k \to B$, for some $k \in {\mathbb N}$\red{.} 
The result of Bulatov and Zhuk can be phrased as follows. 

\begin{theorem}[\cite{BulatovFVConjecture,ZhukFVConjecture}]
    Let $\bB$ be a finite structure. If $\bB$ has a polymorphism which is a \emph{weak near unanimity operation}, i.e., which satisfies for all $x,y \in B$ that 
    $$f(x,\dots,x,y) = f(x,\dots,y,x) = \cdots = f(y,x,\dots,x),$$
    then $\Csp(\bB)$ is in P. 
\end{theorem}
Before Bulatov and Zhuk proved this theorem, it was already known 
that finite structures without a weak near unanimity polymorphism can simulate the three-coloring problem in a very specific way: using primitive positive interpretations. Actually, this hardness condition works for arbitrary (not only finite domain) structures, and we recall it in detail in the following. 

A \emph{primitive positive formula} is a formula $\phi(y_1,\dots,y_k)$ of the form
$$ \exists x_1,\dots,x_n (\psi_1 \wedge \cdots \wedge \psi_m)$$
where $\psi_1,\dots,\psi_m$ are atomic formulas with variables from $x_1,\dots,x_n$ and the free variables $y_1,\dots,y_k$. Note that the equality symbol $=$, the symbol $\top$ for true, and the symbol $\bot$ for false are permitted in atomic formulas. A relation $R \subseteq B^k$ is called \emph{primitively positively definable} in a $\tau$-structure $\bB$ if there exists a primitive positive formula $\phi(y_1,\dots,y_k)$ over the signature $\tau$ such that $R = \{(b_1,\dots,b_k) \mid \bB \models \phi(b_1,\dots,b_k) \}$. 
Suppose that 
$\phi$ is a primitive positive formula 
build with relation symbols from the signature $\tau$ and without the equality symbol. The \emph{canonical database} $\bA$ of $\phi$ 
is the $\tau$-structure whose elements are the (free and existentially quantified) variables of $\phi$, and where $a \in R^{\bA}$, for a relation symbol $R \in \tau$ of arity $k$, 
if $\phi$ contains the conjunct $R(a)$. 
We stress that the following two lemmata also hold for structures $\bB$ with an infinite domain. 

\begin{lemma}[see, e.g., \cite{JBK,Book}]\label{lem:pp-reduce}
    Suppose that $R$ is a relation with a primitive positive definition in $\bB$. Then there is a polynomial-time reduction from $\Csp(\bB,R)$ to 
    $\Csp(\bB)$. 
\end{lemma}

It turns out that primitive positive definability over a structure $\bB$ with a finite domain can be characterised using 
the polymorphisms of $\bB$ (we mention that the result below also holds for countably infinite $\omega$-categorical structures $\bB$~\cite{BodirskyNesetrilJLC}; however, this fact is not needed in our proofs).

\begin{theorem}[see, e.g.,~\cite{BoKaKoRo,Geiger}]\label{thm:ppdef-pol}
Let $\bB$ be a relational structure with a
finite domain. A relation $R$ has a primitive positive definition on $\bB$
if and only if $R$ is preserved by all polymorphisms of $\bB$.
\end{theorem}


If $\bB$ is a $\tau$-structure, and $\bA$ is a $\sigma$-structure, then 
a \emph{primitive positive interpretation} of $\bA$ in $\bB$ is a partial map $I$ from $B^d$ to $A$ such that for each atomic formula $\phi(y_1,\dots,y_k)$ over the signature $\sigma$ there exists a primitive positive formula $\phi_I$ that defines $\big 
\{(b_1,\dots,b_k) \in B^{kd} \mid \bA \models \phi(I(b_1),\dots,I(b_k)) \big \}.$
We refer to $d$ as the \emph{dimension} of $I$. 

\begin{lemma}[see, e.g.,~\cite{Book}]\label{lem:pp-interpret-reduce}
    Suppose that $\bA$ has a primitive positive interpretation in $\bB$. Then there is a polynomial-time reduction from $\Csp(\bA)$ to 
    $\Csp(\bB)$. 
\end{lemma}

It is well-known and easy to prove that primitive positive interpretations can be composed (see, e.g.,~\cite{Book}).

\begin{corollary}\label{cor:pp-interpret-hard}
Suppose that $K_3$ has a primitive positive interpretation in $\bB$. 
Then $\Csp(\bB)$ is NP-hard. 
\end{corollary}

The special case of the result of Bulatov and of Zhuk for \emph{Boolean structures}, i.e., structures with a two-element domain, is of particular importance in the present paper. This case has been known \red{for} much longer, and it is referred to as Schaefer's theorem. In this special case, one can spell out concrete descriptions of the weak near unanimity polymorphisms that imply polynomial-time tractability of the CSP.  

\begin{theorem}[Schaefer's theorem]\label{thm:schaefer}
Let $\bB$ be a structure with a domain of size two.
Then either $\bB$ interprets $K_3$ 
primitively positively, 
or $\bB$ has one of the following weak near unanimity polymorphisms
\begin{itemize}
    \item the binary \emph{minimum or maximum operation}, i.e., 
    an operation $f$ satisfying $f(x,y)=f(y,x)$ and $f(x,x) = x$ for all $x,y \in B$, 
    \item the \emph{ternary majority operation}, i.e., the (unique!) operation $f$ satisfying $f(x,x,y) = f(x,y,x) = f(y,x,x) = x$ for all $x,y \in B$, 
    \item the \emph{ternary minority operation}, i.e., the (unique!) operation $f$ satisfying $f(x,x,y) = f(x,y,x) = f(y,x,x) = y$ for all $x,y \in B$, 
    \item a constant operation, i.e., an operation satisfying $f(x) = f(y)$ for all $x,y \in B$. 
\end{itemize}
In all of these cases, $\Csp(\bB)$ is in P. 
\end{theorem}

A full complexity classification for Boolean CSPs up to logspace reductions can found in~\cite{AllenderSchaefer}; also see~\cite{PPPoset}. 
The following is a well-known fact from linear algebra and will be useful later.  

\begin{lemma}\label{lem:minority-leqs} 
A Boolean relation $R \subseteq \{0,1\}^n$
is preserved by the Boolean minority operation if and only if $R$ is the solution space of a system of linear equalities over the two-element field ${\mathbb F}_2$.  
\end{lemma}

The \emph{infinite-domain tractability conjecture} of Bodirsky and Pinsker from 2011 first appeared in~\cite{BPP-projective-homomorphisms}; the formulation below is equivalent to it by results from~\cite{BKOPP-equations}.  

\begin{conjecture}
    Let $\bB$ be a reduct of a finitely bounded countable homogeneous structure. If 
    $\bB$ does not interpret primitively positively a graph which is homomorphically equivalent to $K_3$, then $\Csp(\bB)$ is in P. 
\end{conjecture}

We will verify a strong form of this conjecture for the ${\mathcal F}$-free orientation problem (Theorem~\ref{thm:main}; we only require that $\bB = H_{\mathcal F}$ does not interpret $K_3$ primitively positively).
A known obstruction for $\bB$ to admit a primitive positive interpretation of $K_3$
is the existence of a so-called \emph{pseudo weak near uanimity polymorphism}, which is a polymorphism \red{$f$} of $\bB$ of arity $k \geq 2$ such that there are endomorphisms $e_1,\dots,e_k$ satisfying for all $x,y \in B$ that
\begin{align}
    e_1(f(x,\dots,x,y)) = e_2(f(x,\dots,y,x)) = \cdots = e_k(f(y,x,\dots,x)).
    \label{eq:pwnu} 
\end{align}
We will verify in Theorem~\ref{thm:main} that if $H_{\mathcal F}$ does not admit a primitive positive interpretation of $K_3$, then it has a pseudo weak near unanimity polymorphism.

\begin{example}
    \label{expl:Rado-pwnu}
    Recall from Section~\ref{sect:MT} that if all tournaments in $\mathcal F$ contain a directed cycle, then $H_{\mathcal F}$ is isomorphic to the Rado graph ${\mathfrak R}$. 
    Using the homogeneity of ${\mathfrak R}$, it is easy to construct an embedding $f \colon {\mathfrak R}^3 \to {\mathfrak R}$.  
    One may find embeddings $e_1,e_2,e_3$ of ${\mathfrak R}$ 
    into itself such that
    $e_1(f(x,x,y)) = e_2(f(x,y,x)) = e_3(f(y,x,x))$, so 
    $\mathfrak R$ has a ternary pseudo \red{weak} near unanimity polymorphism \red{(in this case, we even have a self-embedding $e_0$ of ${\mathfrak R}$ such that we may add $ = e_0(f(x,x,x))$ to the equation, which is why $f$ is called a \emph{pseudo near unanimity}, i.e., we can drop the adjective `weak' in this case).} The same construction works for the Henson graphs. 
\end{example}


\section{The Orientation Completion Problem}
\label{sect:dicho-oc}
We divide this section into two parts. In the first one, we show that for each finite
set of finite tournaments $\mathcal{F}$ there is a Boolean structure whose CSP is
equivalent to the $\mathcal{F}$-free orientation completion problem. This naturally
yields a complexity classification of the $\mathcal F$-free orientation completion 
problem in terms of Schaefer's cases and the constructed Boolean structure.  In the
second part, we observe that these cases reduce to only two possibilities: either
the Boolean structure primitively positively interprets $K_3$, or it is preserved by
the minority polymorphism. In particular, this means that either the $\mathcal F$-free
orientation completion problem reduces (in log-space) to linear equations over
$\mathbb Z_2$, or otherwise,  the $\mathcal F$-free orientation completion problem
is NP-complete. We provide several examples in \cref{sect:examples}.

\subsection{Equivalence to Boolean CSPs}
If $T$ is a tournament with vertex set $\{1,\dots, n\}$, $n \ge 2$, we define
$b_T\in \{0,1\}^{n \choose 2}$ as follows. The entries of $b_T$ will be indexed by $2$-element
subsets $\{i,j\}$ of $\{1,\dots, n\}$ written as $(b_T)_{ij}$. For all $\{i,j\}\subseteq \{1,\dots, n\}$
with $i < j$ we have that $(b_T)_{ij} = 1$ if and only if $E(i,j)$. This coding clearly yields a
bijection between  $\{0,1\}^{n \choose 2}$ and labeled tournaments with vertex set $\{1,\dots, n\}$.
We illustrate this coding in \cref{fig:3-vertices}.


Let $\bB_{\mathcal F}$ be the structure with domain $\{0,1\}$  whose signature contains for every $n \in  \{2,\dots,m_\mathcal F \}$ 
the relation symbol $P_n$ of arity $n \choose 2$ which denotes in ${\bB_{\mathcal F}}$ the relation consisting of all  
$b_T \in \{0,1\}^{n \choose 2}$
such that the tournament $T$ is ${\mathcal F}$-free.
The structure $(\bB_\mathcal{F}, {\bf 0},{\bf 1})$ is the expansion of
$\bB_\mathcal{F}$ by the two unary singleton relations
${\bf 0} := \{0\}$ and ${\bf 1} := \{1\}$.

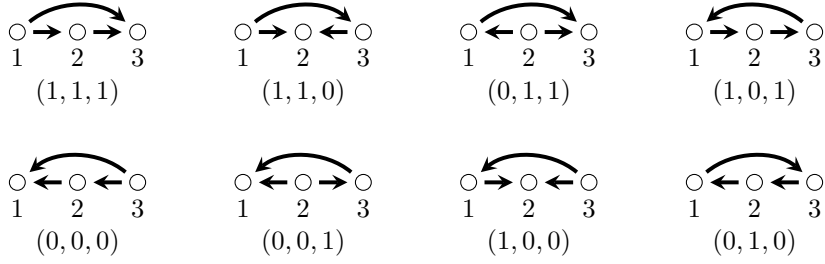
\begin{figure}[ht!]
\centering
\begin{tikzpicture}

  \begin{scope}[xshift = -4.5cm, scale = 0.8]
    \node [vertex, label = below:$1$] (1) at (-1,0) {};
    \node [vertex, label = below:$2$] (2) at (0,0) {};
    \node [vertex, label = below:$3$] (3) at (1,0) {};
    \node (L1) at (0,-1) {$(1,1,1)$};
        
    \draw [arc] (1) to (2);
    \draw [arc] (2) to (3);
    \draw [arc] (1) to [bend left = 40] (3);
  \end{scope}

  \begin{scope}[xshift = -1.5cm, scale = 0.8]
    \node [vertex, label = below:$1$] (1) at (-1,0) {};
    \node [vertex, label = below:$2$] (2) at (0,0) {};
    \node [vertex, label = below:$3$] (3) at (1,0) {};
    \node (L1) at (0,-1) {$(1,1,0)$};
        
    \draw [arc] (1) to (2);
    \draw [arc] (3) to (2);
    \draw [arc] (1) to [bend left = 40] (3);
  \end{scope}

  \begin{scope}[xshift = 1.5cm, scale = 0.8]
    \node [vertex, label = below:$1$] (1) at (-1,0) {};
    \node [vertex, label = below:$2$] (2) at (0,0) {};
    \node [vertex, label = below:$3$] (3) at (1,0) {};
    \node (L1) at (0,-1) {$(0,1,1)$};
        
    \draw [arc] (2) to (1);
    \draw [arc] (2) to (3);
    \draw [arc] (1) to [bend left = 40] (3);
  \end{scope}

  \begin{scope}[xshift = 4.5cm,  scale = 0.8]
    \node [vertex, label = below:$1$] (1) at (-1,0) {};
    \node [vertex, label = below:$2$] (2) at (0,0) {};
    \node [vertex, label = below:$3$] (3) at (1,0) {};
    \node (L1) at (0,-1) {$(1,0,1)$};
        
    \draw [arc] (1) to (2);
    \draw [arc] (2) to (3);
    \draw [arc] (3) to [bend right = 40] (1);
    \end{scope}

    \begin{scope}[xshift = -4.5cm, yshift = -2cm, scale = 0.8]
    \node [vertex, label = below:$1$] (1) at (-1,0) {};
    \node [vertex, label = below:$2$] (2) at (0,0) {};
    \node [vertex, label = below:$3$] (3) at (1,0) {};
    \node (L1) at (0,-1) {$(0,0,0)$};
        
    \draw [arc] (3) to (2);
    \draw [arc] (2) to (1);
    \draw [arc] (3) to [bend right = 40] (1);
  \end{scope}

  \begin{scope}[xshift = -1.5cm, yshift = -2cm, scale = 0.8]
    \node [vertex, label = below:$1$] (1) at (-1,0) {};
    \node [vertex, label = below:$2$] (2) at (0,0) {};
    \node [vertex, label = below:$3$] (3) at (1,0) {};
    \node (L1) at (0,-1) {$(0,0,1)$};
        
    \draw [arc] (2) to (1);
    \draw [arc] (2) to (3);
    \draw [arc] (3) to [bend right = 40] (1);
  \end{scope}

  \begin{scope}[xshift = 1.5cm, yshift = -2cm, scale = 0.8]
    \node [vertex, label = below:$1$] (1) at (-1,0) {};
    \node [vertex, label = below:$2$] (2) at (0,0) {};
    \node [vertex, label = below:$3$] (3) at (1,0) {};
    \node (L1) at (0,-1) {$(1,0,0)$};
        
    \draw [arc] (3) to (2);
    \draw [arc] (1) to (2);
    \draw [arc] (3) to [bend right = 40] (1);
  \end{scope}

  \begin{scope}[xshift = 4.5cm, yshift = -2cm, scale = 0.8]
    \node [vertex, label = below:$1$] (1) at (-1,0) {};
    \node [vertex, label = below:$2$] (2) at (0,0) {};
    \node [vertex, label = below:$3$] (3) at (1,0) {};
    \node (L1) at (0,-1) {$(0,1,0)$};
        
    \draw [arc] (3) to (2);
    \draw [arc] (2) to (1);
    \draw [arc] (1) to [bend left = 40] (3);
  \end{scope}

\end{tikzpicture}
\caption{The eight labeled tournaments on $3$ vertices. The labels correspond
to the associated tuple $(x_{1,2},x_{1,3},x_{2,3})$ where $x_{i,j} = 1$ if 
$(i,j)\in E(G)$, and  $x_{i,j} = 0$ if $(j,i)\in E(G)$ for $1\le i < j \le 3$.}
\label{fig:3-vertices}
\end{figure}

\begin{example}
For the sake of clarity, we provide an explicit description of $\bB_\mathcal{F}$
if $\mathcal{F}$ is a set of tournaments on $3$ vertices.  Firstly, it is easy
to see that the ternary relation $P_3^{\bB_\mathcal{F}}$ is empty if 
$\mathcal{F} = \{T_3,\overrightarrow{C_3}\}$, i.e., there is no
$\mathcal{F}$-free orientation of $K_3$ in this case. If
$\mathcal{F} = \varnothing$,  then $P_3^{\bB_\varnothing} = \{0,1\}^3$, i.e.,
any orientation of $K_3$ is $\mathcal{F}$-free.  If $\mathcal{F} = \{T_3\}$
the relation $P_3^{\bB_\mathcal{F}}$  is the set $\{(1,0,1),(0,1,0)\}$,
since the  $T_3$-free orientations of $K_3$ correspond to both cyclic orientations
of $K_3$.
Finally, if $\mathcal{F} = \overrightarrow{C_3}$, then
$P_3^{\bB_\mathcal{F}} = \{0,1\}^3\setminus \{(1,0,1), (0,1,0)\}$.
\end{example}

A reduction similar to the reduction in the next theorem has been
described in~\cite{BodMot-Unary}.

\begin{theorem}\label{thm:reduction-to-boolean}
    The following statements hold for any finite set of finite tournaments $\mathcal{F}$.
    \begin{enumerate}
        \item There is a polynomial-time reduction from the ${\mathcal F}$-free
            orientation problem to $\Csp(\bB_{\mathcal F})$. 
         \item There is a polynomial-time reduction from the ${\mathcal F}$-free
            orientation completion problem to $\Csp(\bB_{\mathcal F},{\bf 0},{\bf 1})$. 
    \end{enumerate}
\end{theorem}
\begin{proof}
    Let $D$ be a given input digraph of the ${\mathcal F}$-free orientation
    completion problem, and fix an enumeration $(v_1,\dots,v_n)$ of the vertex
    set $V(D)$. Create a variable $x_{i,j}$ for each $i,j \in \{1,\dots,n\}$ with
    $i<j$. 
    Now suppose that $v_{i_1},\dots,v_{i_\ell}$, for $i_1<\cdots<i_\ell$, induce a semicomplete digraph.
    If \red{$\ell \le m_{\mathcal F}$}, then we add the constraint \red{$P_\ell(x_{i_1,i_2},x_{i_1,i_3},\dots,
    x_{i_{\ell-1},i_\ell})$}. Finally, for each pair of vertices $i< j$ we add the constraint
    ${\bf 1}(x_{i,j})$ if $(i,j)\in E(D)$ and $(j,i)\not\in E(D)$; otherwise, if
    $(i,j)\not\in E(D)$ and $(j,i)\in E(D)$, we add the constraint ${\bf 0}(x_{i,j})$.
    Clearly, the resulting instance of $\Csp(\bB_{\mathcal F}, {\bf 0},{\bf 1})$ has
    a solution if and only if $D$ can be completed to an ${\mathcal F}$-free
    oriented graph.
    With similar arguments but omitting the unary constraints ${\bf 0}(x)$ and ${\bf 1}(x)$,
    we obtain a polynomial-time reduction from the $\mathcal{F}$-free orientation
    problem to $\Csp(\bB_\mathcal{F})$.
\end{proof}

In the rest of this section, we show that there is a polynomial-time reduction 
from $\Csp(\bB_\mathcal{F})$ to the $\mathcal{F}$-free orientation completion
problem, and use this reduction to classify the 
complexity of the $\mathcal{F}$-free orientation completion problem. 

Given a digraph $D = (V,E)$ and an edge $(x,y)\in E$, we write $D - (x,y)$ to denote
the digraph $(V, E\setminus\{(x,y)\})$.
Consider a set of tournaments $\mathcal{F}$ and a symmetric edge $xy$ of $D$. 
We say that $xy$ is  \textit{free} in $D$ (with respect to $\mathcal{F}$)
if $D-(x,y)$
and $D-(y,x)$ can be completed to an
$\mathcal{F}$-free oriented graph. We say that a pair $(x,y)$ \textit{forces}
a pair $(u,v)$ in $D$ (with respect to $\mathcal{F}$) if $xy$ and $uv$ are free
symmetric edges in $D$, and every $\mathcal{F}$-free orientation completion of
$D-(y,x)$ contains $(u,v)$ as an oriented edge. In other words, if $xy$
and $uv$ are free edges in $D$, we say that $(x,y)$
forces $(u,v)$ if any orientation completion $D'$ of $D$ such that 
$(x,y),(v,u)\in E(D')$ contains some tournament $T$ of $\mathcal{F}$.
For instance, in the digraph $D_1$ (see \cref{fig:forced-edges}) the pair
$(x,y)$ forces the pair $(u,v)$ with respect to $\{\overrightarrow{C_3}\}$.
Recall that if $\mathcal F$ contains a transitive tournament, we denote by $n_\mathcal F$
the minimum number of vertices of a transitive tournament in $\mathcal F$.

\begin{lemma}\label{lem:forced-edges}
Let $\mathcal F$ be a non-empty finite set of tournaments, and $m$ the minimum number of vertices in a tournament
of $\mathcal F$. In this case, the following statements are equivalent.
\begin{enumerate}
    \item An oriented graph $D$ is $\mathcal F$-free if and
    only if it contains no tournament on $m$ vertices. 
    \item For each digraph $D$ there is an $\mathcal{F}$-free orientation completion of 
    $D$ if and only if every orientation completion of  $D$ is $\mathcal{F}$-free.
    \item For every digraph $D$, if a pair $(x,y)$ forces a pair $(u,v)$ in $D$ with respect
    to $\mathcal{F}$, then $x = u$ and $y = v$.
    \item For every semicomplete digraph $D$ on $m$ vertices, if a pair $(x,y)$ forces a pair $(u,v)$ in $D$ 
    with respect to $\mathcal{F}$, then $x = u$ and $y = v$.
\end{enumerate}
\end{lemma}
\begin{proof}
Suppose that the first item holds. It immediately follows that a digraph $D$ admits
an orientation completion if and only if $D$ contains no semicomplete digraph on $m$
vertices. In this case, any orientation completion of $D$ is $\mathcal F$-free. Thus, the first
item implies the second one. 

Assume the second statement to be true. Since there is a tournament on $m$ vertices in $\mathcal F$, 
it must be the case that no orientation of $K_m$ is $\mathcal{F}$-free. By the choice of $m$
it must be the case that $\mathcal{F}$ contains all tournaments on $m$ vertices up to isomorphism, 
and that $\mathcal F$ contains no tournament on less than $m$ vertices. Thus, 
an oriented graph is $\mathcal F$-free if and only if it contains no tournament on $m$
vertices.

Directly from the definition of ``$(x,y)$ forces $(u,v)$'' one can notice that the negation
of the third statement implies the negation of the second one. Equivalently, the second item
implies the third one, and trivially, the third item implies the fourth one. Finally, we argue
that the fourth item implies the first one by contraposition. So, assuming the first item is
not true, we know that there  must be at least one $\mathcal F$-free tournament on $m$ vertices.
Let $T$ be a tournament in $\mathcal{F}$ with $m$ vertices, and  $(x_1,y_1),\dots, (x_n,y_n)$
be the edges of $T$.  Consider the semicomplete digraph $T^i$ recursively defined as
$T^i: = (V(T), E(T^{i-1})\cup \{(y_i,x_i)\})$,
where $T^0 := T$. In particular, notice
that $T^n$ is the complete graph on $m$ vertices, so $T^n$
admits an $\mathcal{F}$-free orientation. Moreover, by the symmetries of complete graphs, 
for any edge $(x,y)\in T^n$ there is an $\mathcal{F}$-free orientation completion of 
$T^n-(x,y)$. Let $l$ be minimal such that $T^l$ can be completed to
an $\mathcal{F}$-free tournament. Clearly, $x_ly_l$ is a symmetric edge in $T^l$,
and since $l\le n-1$, there is non-symmetric edge in $T^l$, namely, $(x_{l+1},y_{l+1})\in E(T^l)$
and $(y_{l+1},x_{l+1})\not\in E(T^l)$. From these observations, and by the choice of 
$l$, it follows that $x_ly_l$ and $x_{l+1}y_{l+1}$ are (different) symmetric edges in 
$T^{l+1}$, and $(x_{l+1},y_{l+1})$ forces $(y_l,x_l)$. The claim follows.
\end{proof}

It is evident that for every set of tournaments $\mathcal{F}$, every digraph $D$,  and every
symmetric edge $xy\in E(D)$, the edge $(x,y)$ forces itself. In the proof of the following
lemma, we will implicitly use the following observations several times:
\begin{enumerate}
    \item If $(x,y)$ forces $(u,v)$, then $(v,u)$ forces $(y,x)$. 
    \item If $(x,y)$ forces $(u,v)$ and $(u,v)$ forces $(a,b)$, then 
    $(x,y)$ forces $(a,b)$.
    \item If $\varphi\colon D\to D'$ is a homomorphism, and $(x,y)$ forces
    $(u,v)$ in $D$, then $(\varphi(x),\varphi(y))$ forces $(\varphi(u),\varphi(v))$
    in $D'$. 
\end{enumerate}

The following lemma shows that given a digraph $D$ with a pair $(x,y)$ that forces
a pair $(u,v)$, we can construct a digraph $D'$ with a pair $(x', y')$ that
forces a pair $(u',v')$, and the latter also forces the former. Moreover, $D'$
can be chosen in such \red{a way that} the vertices $x',y'$ are ``far apart'' from 
the pair $u',v'$. To this end, we consider the following notion of distance.
Given a pair of vertices $x,y$ in a connected digraph $D$, we denote by  $d(x,y)$
the distance between $x$ and \red{$y$} in the underlying graph $u(D)$.  That is,
the number of edges in a shortest path between $x$ and $y$ in $u(D)$, \red{which can traverse edges ignoring their direction}.
The previously mentioned construction is described in the proof of the following lemma, 
and illustrated in \cref{fig:forced-edges}. 

\begin{lemma}\label{lem:mutual-implication}
    The following statements are equivalent for a finite set of finite tournaments
    $\mathcal{F}$.
    \begin{enumerate}
        \item There is a digraph $D$ with two pairs of vertices $(x,y)$ and
        $(u,v)$ such that $(x,y)$ forces $(u,v)$, and $(x,y)\neq (u,v)$.
        \item There is a digraph $D$ with two pairs of vertices $(x,y)$ and
        $(u,v)$ such that $(x,y)$ forces $(u,v)$, and $|\{x,y,u,v\}| = 4$.
        \item For every positive integer $k$, there is a digraph $D$ with two pairs
        of vertices $(x,y)$ and $(u,v)$ such that $(x,y)$ and $(u,v)$ force each other,
        and $d(a,b) \ge k$ for \red{any} $a\in\{x,y\}$ and $b\in \{u,v\}$.
    \end{enumerate}
\end{lemma}
\begin{proof}
    It is evident that the third statement implies the first one. 
    Now, we prove that the first item implies the second one. Suppose that 
    $(x,y) \neq (u,v)$. If $|\{x,y,u,v\}| = 4$, then there is nothing left to 
    prove. So suppose that $|\{x,y,u,v\}| = 3$. Notice that up to symmetry, there
    are two cases to consider: when $y = u$, and when $y = v$. We consider the
    latter one. Consider two copies of $D$, $D_1$ and $D_2$, where $(x_i,y_i)$
    forces $(u_i,y_i)$ in $D_i$ for each $i\in\{1,2\}$. In this case, let $D'$ be 
    the digraph obtained from the disjoint union of $D_1$ with $D_2$ after 
    identifying $y_1$ with $u_2$ and $u_1$ with $y_2$. Using the enumerated 
    observations preceding this lemma,  we conclude that $(x_1,y_1)$ forces 
    $(y_2,x_2)$ in $D'$. The case when $y = u$ follows with a similar construction.

    Finally, we show that the second statement implies the third one.
    Let $D$, $x$, $y$, $u$, and $v$ be as in the second statement, and $k\ge 2$.
    Consider $k$ copies $D_1,\dots, D_k$ of $D$ where $(x_i,y_i)$ forces $(u_i,v_i)$ in
    $D_i$ for each $i\in\{1,\dots, k\}$. It is not hard to notice that
    by considering the disjoint union $D_1 + \cdots + D_k$, and identifying $u_i$ 
    with $x_{i+1}$, and $v_i$ with $y_i$ for $i\in[k-1]$, we obtain a digraph $D'$
    where $(x_1, x_2)$ forces $(u_k,v_k)$, and $d(a,b) \ge k$ for \red{any} $a\in\{x_1,y_1\}$
    and $b\in\{u_k,v_k\}$. Finally,  consider $D'$ together with a disjoint copy
    $D''$ of itself, and the following identifications $u_k' \sim x_1''$, $u_k'' \sim x_1'$,
    $v_k' \sim y_1''$, and $v_k'' \sim y_1'$. It is not hard to see that we obtain a
    digraph $(D' + D'')/_{\sim}$ with two pairs of vertices $(x_1', y_1')$ and $(u_1', v_1')$
    that force each other and  $d(a,b) \ge k$ for \red{any} $a\in\{x_1',y_1'\}$ and $b\in \{u_1',v_1'\}$. 
    The lemma is now proved. 
    \end{proof}

\begin{figure}[ht!]
\centering
\begin{tikzpicture}[scale = 0.8]

\begin{scope}[xshift = -6cm]
    \node [vertex, label = below:{$u$}] (u) at (-1,0) {};
    \node [vertex, label = above:{$x$}] (x) at (-1,2) {};
    \node [vertex, label = above:{$y = v$}] (v) at (1,2) {};
      
    \foreach \from/\to in {u/v, x/v}     
    \draw [edge] (\from) to (\to);

    \draw [arc] (u) to (x);

    \node (L1) at (0,-1) {$D_1$};
    
  \end{scope}

  \begin{scope}[xshift = 0cm]
    \node [vertex, label = above:{$x$}] (x) at (-1,2) {};
    \node [vertex, label = above:{$y$}] (y) at (1,2) {};
    
    \node [vertex, label = below:{$u$}] (u) at (-1,0) {};
    \node [vertex, label = below:{$v$}] (v) at (1,0) {};
      
    \foreach \from/\to in {x/y, u/v, u/y}     
    \draw [edge] (\from) to (\to);

    \foreach \from/\to in {u/x, y/v}     
    \draw [arc] (\from) to (\to);

    \node (L1) at (0,-1) {$D_2$};
    
  \end{scope}

  \begin{scope}[xshift = 6cm]
    \node [vertex, label = above:{$x$}] (x) at (-1,4) {};
    \node [vertex, label = above:{$y$}] (y) at (1,4) {};
    
    \node [vertex, label = below:{$u$}] (u) at (-1,0) {};
    \node [vertex, label = below:{$v$}] (v) at (1,0) {};
    
    \node [vertex] (1) at (-2.6,2) {};
    \node [vertex] (2) at (-.6,2) {};
    \node [vertex] (3) at (.6,2) {};
    \node [vertex] (4) at (2.6,2) {};
    
    \foreach \from/\to in {1/2, 1/y, x/4, 3/4, 3/v, u/2, x/y, u/v}     
    \draw [edge] (\from) to (\to);

    \foreach \from/\to in {1/x, y/2, u/1, 2/v, 3/u, x/3, v/4, 4/y}     
    \draw [arc] (\from) to (\to);

    \node (L1) at (0,-1) {$D_3$};
    
  \end{scope}
  
\end{tikzpicture}
 \caption{Three digraphs $D_1$, $D_2$, and $D_3$ where in each case, the
 pair $(x,y)$ forces the pair $(u,v)$ with respect to $\{\overrightarrow{C_3}\}$. 
 In $D_1$, the cardinality of $\{x,y,u,v\}$ is $3$. In $D_2$,
 $|\{x,y,u,v\}| = 4$, and it is obtained from 
 $D_1$ as in the proof of \cref{lem:mutual-implication}. Finally, in $D_3$,
 $d(a,b) \ge 2$ for \red{any} $a\in\{x,y\}$ and $b\in \{u,v\}$, and $D_3$
 is obtained from $D_2$ as in the proof of \cref{lem:mutual-implication} for
 $k = 2$.}
\label{fig:forced-edges}
\end{figure}
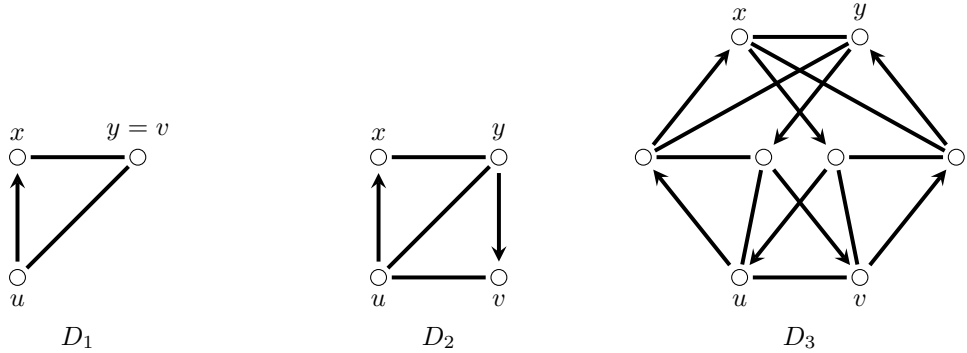

Recall that the $\mathcal{F}$-free orientation completion problem and $\Csp(D_\mathcal{F},
U)$ are trivially polynomial-time equivalent.  We will see that
there is a primitive positive interpretation of $(\bB_\mathcal{F}, {\bf 0},{\bf 1})$ in $(D_\mathcal{F},
U)$. This will yield  a polynomial-time reduction from $\Csp(\bB_\mathcal{F},
{\bf 0},{\bf 1})$ to $\Csp(D_\mathcal{F}, U)$, and thus, also to the $\mathcal{F}$-free
orientation completion problem. To do so, we first consider the $4$-ary relation  $S_4$ defined by
\begin{align}
S_4(x_1, x_2, x_3, x_4) \Leftrightarrow
\big ((E(x_1,x_2) \land E(x_3,x_4))
\lor (E(x_2,x_1) \land E(x_4,x_3)) \big ) .
\label{eq:S4}
\end{align}
Intuitively, $S_4$ encodes that the first and last pair of vertices are adjacent, and
both edges have the same direction. 

\begin{lemma}
\label{lem:pp-definition-S}
For every non-empty finite set of tournaments $\mathcal{F}$, there is a primitive positive
definition of $S_4$ in $(D_\mathcal{F}, U)$.
\end{lemma}
\begin{proof}
    By \cref{lem:forced-edges}, and by the third part of \cref{lem:mutual-implication},
    there is a digraph $D$ with two pairs of vertices $(x,y)$ and $(u,v)$ such that
    $(x,y)$ and $(u,v)$ force each other, and $d(a,b) \ge 4$ for \red{any} $a\in\{x,y\}$ and
    $b\in \{u,v\}$. Interpret $D$ as an $\{E,U\}$-structure $D_U$, where $(x,y)\in E(D_U)$
    if and only if $(x,y)\in E(D)$ and $(y,x)\not\in E(D)$, and $(x,y)\in U(D_U)$ if and only
    if $(x,y)\in E(D)$ and $(y,x)\in E(D)$. In other words, the interpretation of $U$ in $D'$
    corresponds to the symmetric edges of $D$, and the interpretation of $E$ corresponds
    to the anti-symmetric edges of $D$ \red{--- we point out that, contrary to the case of
    $(D_\mathcal F, U)$, in the case of $D_U$ the interpretation 
    of $U$ does not correspond to the symmetric closure of the interpretation of $E$}.
    Let $\phi(x_1,\dots, x_n)$ be the canonical conjunctive
    query of $D_U$, where $x_1$ corresponds to $x$, $x_2$ to $y$, $x_3$ to $u$, and $x_4$ to $v$.
    With this setting, if $\phi_S(x_1,x_2,x_3,x_4)$ is the primitive positive-formula $\exists x_5,
    \dots, x_n\phi(x_1,\dots, x_n)$, then $\phi_S$ implies $S_4$. Now, we briefly argue that
    $S_4$ also implies $\phi_S$. Let $y_1,\dots, y_4$ be four vertices of \red{$D_\mathcal F$} such
    that $S_4(y_1,y_2,y_3,y_4)$, and without loss of generality assume that $(y_1,y_2),(y_3,y_4)\in
    E(D_\mathcal F)$. Suppose that $|\{y_1,y_2,y_3,y_4\}| = 4$ and let $D'$ be \red{an} $\mathcal F$-free
    orientation of $D$ where $(x_1,x_2),(x_3,x_4)\in E(D')$. For each $i,j\in\{1,2,3,4\}$, if there
    is an edge $(y_i,y_j)$ in $D_\mathcal F$, then add an edge $(x_i,x_j)$ to $D'$ obtaining an
    oriented graph $D''$. Notice that since $d(x_i,x_j) \ge 4$ for $i\in\{1,2\}$ and $j\in\{3,4\}$,
    any triangle of $D''$ is a triangle of $D'$ and so, $D''$ is an $\mathcal F$-free oriented graph. 
    Thus, there is partial automorphism that maps $y_i\mapsto x_i$ for $i\in\{1,2,3,4\}$, and by homogeneity
    of $D_\mathcal F$ this can be extended to an automorphism $f\colon D_\mathcal F\to D_\mathcal F$. 
    Since every edge of $D'$ is an edge of $D''$ and $\phi_S$ is a primitive positive formula, it is the case
    that $\phi_S$ is true of $x_1,x_2,x_3,x_4$ in $D_\mathcal F$. Therefore, since primitive positive formulas are
    preserved by automorphisms, and $f^{-1}(x_i) = y_i$ for $i\in\{1,2,3,4\}$, we conclude that
    $D_\mathcal F\models \phi_S(y_1,y_2,y_3,y_4)$. Finally, the cases when $|\{y_1,y_2,y_3,y_4\}|\in
    \{2,3\}$ follow similarly, but instead of obtaining
    $D''$ from $D'$ by adding new edges, we obtain $D''$ from $D'$ by identifying $x_i$ with $x_j$ whenever
    $y_i = y_j$. Again, the observation that $D''$ is a $\mathcal F$-free oriented graph follows from the
    fact that $d(x_i,x_j)\ge 4$ for $i\in\{1,2\}$ and $j\in\{3,4\}$. With the corresponding identifications,
    we obtain a partial automorphism of $D_\mathcal F $ that defined by mapping $y_i\mapsto x_i$ for $i\in\{1,2,3,4\}$.
    Finally, using the fact that primitive positive formulas are preserved under homomorphisms (and $D''$ is a homomorphic
    image of $D'$) and under automorphisms, we conclude that $D_\mathcal F\models \phi_S(y_1,y_2,y_3,y_4)$.
\end{proof}

Now, we show that using the relation $S_4$ we can primitively positively interpret
$(\bB_\mathcal{F}, {\bf 0},{\bf 1})$ in $(D_\mathcal F, U, S_4)$. Also, recall that $H_\mathcal F$
is the underlying graph of $D_\mathcal F$.
We also see that we can primitively positively interpret
$\bB_\mathcal{F}$ in $(H_\mathcal F, S_4)$.

\begin{lemma}
\label{lem:interpret-B}
For every finite set of finite tournaments $\mathcal{F}$, there is a primitive 
positive interpretation of $(\bB_\mathcal{F},{\bf 0},{\bf 1})$ in $(D_{\mathcal F},U,S_4)$,
and a primitive positive interpretation of $\bB_\mathcal F$ in $(H_{\mathcal F},S_4)$.
\end{lemma}
\begin{proof}
    We first consider the primitive positive interpretation of $(\bB_\mathcal{F},{\bf 0},{\bf 1})$ in 
    $(D_{\mathcal F},U,S_4)$. The dimension of the interpretation is $2$, and the domain 
    formula is $\top_I(x,y) := U(x,y)$. Equality ${=_I}(x_1,y_1, x_2,y_2)$ is defined by
     $S_4(x_1,y_1, x_2,y_2)$, and the unary relations ${\bf 0}_I(x,y)$ and ${\bf 0}_I(x,y)$ are
    defined by $E(y,x)$ and $E(x,y)$, respectively. 
    Finally, for each positive integer $n$, the $2{n \choose 2}$-ary relation
    \[ \delta_{P_n}(x_{1,2},y_{1,2},x_{1,3},y_{1,3},\dots, x_{n-1,n},y_{n-1,n}) \]
    expresses that there  are $n$ vertices $k_1,\dots, k_n$
    such that:
    \begin{enumerate}
        \item  $U(k_i,k_j)$ for each $i,j\in [n]$ \red{(i.e., the vertices
        $k_1,\dots, k_n$ must induce a tournament in $D_\mathcal F$, because $D_{\mathcal F}$ is an oriented graph)}, and
        \item for each pair $i < j$ the $4$-tuple $(x_{i,j},y_{i,j},k_i,k_j)$ belongs to 
        $S_4$, i.e., the edges $k_ik_j$ and $x_{i,j} y_{i,j}$ have the same orientation
        in $D_\mathcal F$.
    \end{enumerate}
    The same interpretation without the defining formulas $\delta_0$ and $\delta_1$
    yield a primitive positive interpretation of $\bB_\mathcal F$ in $(H_\mathcal F, S_4)$. Notice 
    that in this case, the first item means that the vertices
    $k_1,\dots, k_n$ induce a clique in $H_\mathcal F$.
\end{proof}

Using these two lemmas, we can easily prove the following statement. 

\begin{proposition}
\label{prop:DFU-interprets-BF01}
    Let $\mathcal F$ be a finite set  of finite tournaments. If $\mathcal F$ is not
    the empty set, then there is primitive positive interpretation of
    $(\bB_F,{\bf 0},{\bf 1})$ in $(D_\mathcal{F}, U)$. 
\end{proposition}
\begin{proof}
    By \cref{lem:pp-definition-S}, $S_4$ has a primitive positive definition in 
    $(D_\mathcal{F}, U)$, and by \cref{lem:interpret-B},
    there is primitive positive interpretation of $(\bB_\mathcal F, {\bf 0},{\bf 1})$ in
    $(D_\mathcal{F}, U,S_4)$. 
\end{proof}

\cref{thm:reduction-to-boolean} asserts that the $\mathcal F$-free orientation completion
problem reduces in polynomial-time to $\Csp(\bB_\mathcal F, {\bf 0},{\bf 1})$. 
Also, as mentioned in \cref{sect:MT}, the $\mathcal F$-free orientation
completion problem and $\Csp(D_\mathcal F, U)$ are polynomial-time
equivalent. Finally, \cref{prop:DFU-interprets-BF01} together with
\cref{lem:pp-interpret-reduce} show that $\Csp(B_\mathcal F,{\bf 0},{\bf 1})$
reduces in polynomial-time to $\Csp(D_\mathcal F, U)$. Thus, the following statement
is proved by the arguments in this paragraph.

\begin{theorem}\label{thm:or-comp-CSP}
    The following problems are polynomial-time equivalent
    for each set finite set of finite tournaments $\mathcal F$.
    \begin{enumerate}
        \item The $\mathcal F$-free orientation completion problem.
        \item $\Csp(\bB_\mathcal{F},{\bf 0},{\bf 1})$.
        \item $\Csp(D_\mathcal F, U)$.
    \end{enumerate}
\end{theorem}

If a Boolean structure $\bB$ contains no constant endomorphism, then it is a core, 
and hence $\Csp(\bB)$ and $\Csp(\bB,{\bf 0},{\bf 1})$ are polynomial-time equivalent (see, e.g.,~\cite{Book}).
Notice that for a set of tournaments $\mathcal{F}$ the structure $\bB_\mathcal{F}$
contains a constant endomorphism if and only if for each $n\le m_\mathcal F$ either 
$P_n$ is empty or $P_n$ contains  both constant tuples \red{$(0,\dots, 0)$ and $(1,\dots, 1)$}.
In particular, if $\mathcal F$ contains a transitive tournament and there is
at least one $\mathcal F$-free tournament on tournament on $n_\mathcal F$ vertices,
then $P_{n_\mathcal F}$ is neither empty nor contains a constant tuple, and thus, $\bB_\mathcal F$ 
contains no constant endomorphism. With these arguments in mind, the following statement
 is an immediate implication of \cref{thm:or-comp-CSP}.

 \begin{corollary}\label{thm:or-comp-CSP-no-Tk}
    Let $\mathcal F$ be a finite set of finite tournaments that contains a transitive tournament. 
    If  there is at least one $\mathcal F$-free tournament with $n_\mathcal F$-vertices, then
    the following problems are polynomial-time equivalent.
    \begin{enumerate}
        \item The $\mathcal F$-free orientation completion problem.
        \item $\Csp(\bB_\mathcal{F},{\bf 0},{\bf 1})$.
        \item $\Csp(\bB_\mathcal F)$.
        \item $\Csp(D_\mathcal F, U)$.
    \end{enumerate}
\end{corollary}

\subsection{Complexity Classification}

\cref{thm:or-comp-CSP} together with Schaefer's theorem yield a
classification of the $\mathcal{F}$-free orientation completion problem
in terms of the Boolean structure $(\bB_\mathcal F,{\bf 0},{\bf 1})$. In this
section, we see that if $(\bB_F,{\bf 0},{\bf 1})$ does not primitively positively
interpret $K_3$, then it is preserved by the Boolean minority operation, or by a 
constant operation.

\begin{lemma}\label{lem:Tn-collapse}
    Let $\mathcal F$ be a finite set of finite tournaments. The following statements are equivalent
    for each positive integer $n \le m_\mathcal F$ such that $T_n$ is $\mathcal F$-free.
    \begin{enumerate}
        \item Every tournament on $n$ vertices is $\mathcal F$-free.
        \item $P_n = \{0,1\}^{n \choose 2}$.
        \item $P_n$ is preserved by the minimum operation.
        \item $P_n$ is preserved by the maximum operation.
        \item $P_n$ is preserved by the majority operation. 
        \item $P_n$ is preserved by the minority operation.
    \end{enumerate}
\end{lemma}
\begin{proof}
    The first two items are clearly equivalent, and the second item implies 3--6.
    Denote by $b_{\bf 0}$ (resp.\ $b_{\bf 1}$) the constant $0$ (resp.\ constant $1$) 
    tuples of arity $n \choose 2$. Since $T_n$ is $\mathcal F$-free, the tuples $b_{\bf 0}$
    and $b_{\bf 1} $ belong to $P_n$.
    It is not hard to notice that for any pair of tuples $b,b'\in P_n$
    the equalities $\majority(b_{\bf 1}, b, b') = \max(b, b')$, and 
    $\majority(b_{\bf 0}, b, b') = \min(b, b')$ hold. Thus, if $P_n$ is preserved
    by the majority operation, then it is preserved by the minimum and the maximum operations. 

    To conclude the proof, we show that each of the statements 3, 4, 6 imply the first two. 
    Suppose that $P_n$ is preserved by the minimum operation. For $i,j\in [n]$ with $i < j$,
    we denote by $b^{ij}$ the tuple where $(b^{ij})_{kl} = 1$ if and only if $i = k$ and $j = l$.
    We show that each $b^{ij}$ belongs to $P_n$. Given $i < j$ consider the
    following  permutations of $T_n$ where the edge set is defined by the linear ordering of $[n]$:
    \begin{align*}
    T^1 & :=(n,n-1, \dots, j+1, j-1,j-2, \dots, i, j, i-1, i-2,\dots 1), \\ \text{ and } \quad 
    T^2 & :=(n,n-1, \dots, j+1, i, j, j-1,\dots, i+1, i-1, i-2, \dots, 1).
    \end{align*}
    With a simple computation of the minimum operation one can notice that  $\min(b_{T^1},b_{T^2}) = b^{ij}$.
    Since $T_n$ is $\mathcal F$-free, it follows that $b_{T^1},b_{T^2}
    \in P_n$, and so $b^{ij}\in P_n$. It is not hard to notice that the tuple $b^{ij}$ encode
    those tournaments obtained from $T_n$ be reversing the orientation of one arc. Similarly, the
    tuples $c^{ij}$ where $(c^{ij})_{kl} = 0$ if and only if $i = k$ and $l = j$, encode
    the same family of tournaments (up to isomorphism). Thus, it is also the case
    that for each $i,j\in [n]$ with $1\le i < j \le n$, all tuples $c^{ij}$ belong to $P_n$.
    Finally, consider a set of pairs
    $\{(i_1,j_1), \dots, (i_l,j_l)\}\subseteq [n]^2$ with $i_k < j_k$ for each $k\in [l]$.
    By   composing the minimum operation as $\min(c^{i_1j_1},\min(c^{i_2j_2},
    \dots ))$, we obtain a tuple $b$ where $b_{ij} = 0$ if and only if $(i,j) = (i_k,j_k)$
    for some $k \in [l]$. Thus, we conclude that each $b\in  \{0,1\}^{n \choose 2}$ belongs to $P_n$,
    i.e., $P_n = \{0,1\}^{n \choose 2}$, and so, the third item implies the second one. The case when
    $\mathcal{F}$ is preserved by the maximum operation follows with dual arguments. 

    Finally, suppose that $P_n$ is preserved by the minority operation. For a pair of tuples
    $b,b'$ we denote by $b + b'$ the coordinate-wise addition modulo $2$, and notice that
    $\minority(b_{\bf 0},b,b') = b + b'$. Thus, since $P_n$ is preserved
    by the minority operation, and $b_{\bf 0}\in P_n$, we conclude that $P_n$ is closed
    under addition modulo $2$. Evidently, every tuple $b \in \{0,1\}^{n\choose 2}$
    can be expressed as a sum of tuples of the form $b^{ij}$ (introduced in the previous
    paragraph). Hence, it suffices to show that $b^{ij}\in P_n$ for each pair $i<j$. To do
    so, consider the  following  permutations of $T_n$ where the edge set is define by the
    linear ordering of $[n]$:
    \begin{align*}
    T^1 & :=(1, \dots, i-1, i, i+2, \dots, j, i+1, j+1, j+2, \dots, n), \\ \text{ and } \quad 
    T^2 & :=(1, \dots, i-1, i+1, i+2, \dots, j, i, j+1, j+2, \dots, n).
    \end{align*}
    It is not hard to notice that $b^{ij} = b_{T^1} + b_{T^2}$, and since  $T_n$ 
    is $\mathcal F$-free and $P_n$ is closed under addition modulo $2$, we conclude
    that $b^{ij} \in P_n$. So, by the arguments above we conclude that $P_n = \{0,1\}^{n
    \choose 2}$. The equivalence between 1--6 is now proved. 
\end{proof}

Recall that $m_\mathcal F$ denotes the maximum number of vertices of a tournament in 
$\mathcal F$. 

\begin{lemma}\label{lem:collapse}
    Let $\mathcal F$ be a finite set of finite tournaments.  The following statements are 
    equivalent for each positive integer $n \le m_\mathcal F$. 
    \begin{enumerate}
        \item Either all tournaments on $n$ vertices are $\mathcal F$-free, or no tournament on
        $n$ vertices is $\mathcal F$-free.
        \item Either $P_n = \varnothing$ or $P_n = \{0,1\}^{n \choose 2}$.
        \item $P_n$ is preserved by the minimum operation.
        \item $P_n$ is preserved by the maximum operation.
        \item $P_n$ is preserved by the majority operation.
    \end{enumerate}
\end{lemma}
\begin{proof}
    It is evident that the first two items are equivalent, and that each of these implies
    the rest. We show that each of 3--5 imply the first two. To do so, we will see
    that in each case,  if $P_n \neq \varnothing$, then $P_n$ contains the constant tuple
    $b_{\bf 1}$ (where all entries are $1$),  i.e., $T_n$ is $\mathcal F$-free. We will thus
    conclude by \cref{lem:Tn-collapse} that $P_ n =  \{0,1\}^{n \choose 2}$.
    To begin with, suppose that  $P_n$ is preserved by the maximum operation, and that there
    is some tuple  $b_T\in P_n$ for some $\mathcal F$-free tournament $T$ with vertex set $[n]$. 
    If $b_T = b_{\bf 1}$, there is nothing left to prove, so suppose $b_{ij} = 0$ for some
    $1\le i < j \le n$. Let $T'$ be the permutation of $T$ obtained from transposing
    $i$ with $j$. Since $(b_T)_{ij} = 0$ and $(b_{T'})_{ij} = 1$, it follows that  the
    number of coordinates which equal $1$ in
    $\max(b_T, b_{T'})$ is strictly larger than those which equal $1$ in $b_T$. 
    Since $P_n$ is preserved by maximum operation, we can iterate this procedure to see
    that $b_{\bf 1} \in P_n$. Thus $T_n$ is $\mathcal F$-free and so, using 
    \cref{lem:Tn-collapse} we conclude that $P_n = \{0,1\}^{n \choose 2}$.
    The case when $P_n$ is preserved by the minimum operation follows from  dual arguments. 
    
    Finally, suppose that $P_n$ is preserved by majority, and that there is some $\mathcal F$-free
    tournament with vertex set $[n]$, i.e., $b_T\in P_n$. We begin by showing that there is a tuple
    $b_{T'}\in P_n$ such that $(b_{T'})_{1j} = 1$ for all $1 < j \le n$. Let $l$ be the maximum
    outdegree of $T$, and suppose $l < n-1$ (otherwise, there is nothing left to prove). 
    Consider two permutations $T^1$ and $T^2$ of  $T$ such that for $i\in \{1,2\}$, the vertex $1$ is the vertex
    of largest outdegree of $T^i$ and $\{i+1, \dots, i+l\}$ are its outneighbours. Consider a
    third labeling $T^3$ such that $(1,2), (1,i+1) \in E(T^3)$. Clearly, if $b = \majority(b_{T^1},b_{T^2},b_{T^3})$,
    then $b_{1j} = 1$ for all $2\le j \le l+1$. Since $P_n$ is preserved by the majority operation, we can
    proceed inductively to find a tuple $b \in P_n$ such that $b_{1j} = 1$ for all $1 < j \le n$.
    With a similar finite inductive argument over $k \in [n]$, we can find a tuple $b \in P_n$ such that
    $b_{ij} = 1$ for all $i\le k$ and $j > i$. Thus, we conclude that $b_{\bf 1} \in P_n$ and so, 
    $T_n$ is  $F$-free. Hence, by \cref{lem:Tn-collapse}, we conclude that $P_n = \{0,1\}^{n \choose 2}$.
    The claim follows.
\end{proof}

Building on \cref{lem:collapse}, we prove the following statement.

\begin{lemma}\label{lem:minority-BF}
    Let $\mathcal F$ be a finite set of finite tournaments. If
    $\bB_{\mathcal F}$ does not interpret $K_3$ primitively positively,
    then $\bB_{\mathcal F}$ is preserved by the Boolean minority operation or a constant operation. 
\end{lemma}
\begin{proof}
    If $\bB_{\mathcal F}$ does not interpret $K_3$ primitively positively, then, by Schaefer's theorem, 
    $\bB_\mathcal F$ is preserved by the minimum, the maximum, the majority, the minority, or the constant
    operation. If either of the last two cases holds, the claim is proved. Otherwise, suppose that 
    $\bB_\mathcal F$ is preserved by the minimum, the maximum, or the majority operation. 
    Then, for each $n\le m_\mathcal F$ the relation $P_n$ is preserved by one of these operations. 
    By \cref{lem:collapse}, we conclude that for each $n \le m_\mathcal F$ either $P_n =\varnothing$
    or $P_n = \{0,1\}^{n \choose 2}$. Hence, each $P_n$ is trivially preserved by any constant operation,
    and so, $\bB_{\mathcal F}$ is preserved by a constant operation. 
\end{proof}


We are now ready to state the proposed classification for the complexity of $\mathcal F$-free
orientation completion problems.

\begin{theorem}\label{thm:or-comp-classification}
    Let ${\mathcal F}$ be a finite set of finite tournaments. 
    Then exactly one of the following two cases applies.
    \begin{itemize}
        \item $K_3$ has a primitive positive interpretation in $({\mathfrak B}_{\mathcal F},\bf{0},\bf{1})$ and in
        $(D_{\mathcal F},U)$. In this case, $\Csp(D_{\mathcal F},U)$ and the
        ${\mathcal F}$-free orientation completion problem are NP-complete.  
        \item $(B_\mathcal F, \bf 0, \bf 1)$ has the the minority operation as polymporphism, and 
        $(D_{\mathcal F},U)$ has a ternary pseudo \red{weak} near unanimity polymorphism \red{(see Equation~\eqref{eq:pwnu})}. In this case, 
        $\Csp(D_{\mathcal F},U)$ and the ${\mathcal F}$-free orientation \red{completion} problem are in P. 
    \end{itemize} 
\end{theorem}
\begin{proof}
    The two cases are mutually disjoint: if $(B_\mathcal F, \bf 0, \bf 1)$ is has the minority operation as polymorphism,
    then every structure with a first-order interpretation in it has such a polymorphism as well
    (see, e.g., Corollary 6.5.16 in~\cite{Book}), but $K_3$ does not have such a polymorphism (see, e.g., Proposition 6.1.43  in~\cite{Book}). 
    
    Now we see that one of the two cases holds, and first suppose that $K_3$ has a primitive positive interpretation in
    $(\bB_{\mathcal F},{\bf 0},{\bf 1})$. Since $(\bB_{\mathcal F},{\bf 0},{\bf 1})$ has a primitive positive interpretation in
    $(D_{\mathcal F},U)$, by composing these interpretations, we obtain a  primitive positive interpretation of $K_3$ in
    $(\bB_{\mathcal F}, {\bf 0},{\bf 1})$. The NP-hardness of $\Csp(D_{\mathcal F},U)$ and the ${\mathcal F}$-free orientation completion
    problem now follows from Lemma~\ref{lem:pp-interpret-reduce}.

    Otherwise, if $K_3$ does not have a primitive positive interpretation in $(\bB_{\mathcal F},{\bf 0},{\bf 1})$, then Schaefer's theorem 
    (Theorem~\ref{thm:schaefer}) implies that $\Csp(\bB_\mathcal{F},{\bf 0}, {\bf 1})$ is in P, and 
    $(\bB_\mathcal{F},{\bf 0}, {\bf 1})$ has a polymorphism which is a constant operation, or the minimum, maximum, 
    majority, or minority operation.
    \cref{lem:minority-BF}, and the obvious fact that $(\bB_\mathcal{F},{\bf 0}, {\bf 1})$ cannot be preserved by constant operations, imply that
    $(\bB_\mathcal{F},{\bf 0}, {\bf 1})$ has a minority polymorphism. In this case we construct the ternary pseudo \red{weak} near unanimity polymorphism of $(D_{\mathcal F},U)$ as follows. Consider the digraph 
    with domain $V^3$ and edge set 
    $$\big \{((u_1,u_2,u_3),(v_1,v_2,v_3)) \mid (u_1 ,v_1),(u_2,v_2),(u_3,v_3) \in U,~|\{i \mid (u_i,v_i) \in E\}| \in \{0,2\} \big \}.$$ 
    Since $\bB_\mathcal{F}$ has a minority polymorphism, every finite subgraph $F$ of this graph has an embedding $h$ to $D_{\mathcal F}$,
    and by the homogeneity of $D_{\mathcal F}$ there exist automorphisms $e_1,e_2$
    of $D_{\mathcal F}$
    such that for all $x,y \in V$ with 
    $(x,x,y),(x,y,x),(y,x,x) \in V(F)$ we have $e_1(h(x,x,y) = e_2(h(x,y,x)) = h(y,x,x)$.
    The existence of a ternary pseudo weak near unanimity polymorphism can be shown as  Proposition 6.6 in~\cite{BPP-projective-homomorphisms}.
    The polynomial-time tractability of $\Csp(D_{\mathcal F},U)$ and the
        ${\mathcal F}$-free orientation completion problem now follows from the polynomial-time tractability of $\Csp(\bB_\mathcal{F},{\bf 0}, {\bf 1})$ via Theorem~\ref{thm:reduction-to-boolean}.
\end{proof}


\section{Symmetries}
\label{sect:sym}
Recall that, given a set of tournaments $\mathcal{F}$, we denote by $D_\mathcal{F}$ the
countable universal homogeneous $\mathcal{F}$-free digraph, and by $H_\mathcal{F}$ its 
its underlying graph.
In order to classify the complexity of the $\mathcal F$-free orientation problem,
it will be highly useful to understand the symmetries of $H_\mathcal F$ in terms
of the symmetries of $D_\mathcal F$. Specifically, we use the description
of the automorphism group of $H_\mathcal F$ in terms of the automorphism group of
$D_\mathcal F$ proposed by Agarwal and Kompatscher~\cite{AgarwalKompatscher}.
To provide examples of these symmetries we consider the four non-isomorphic
tournaments on $4$ vertices $T_4$, $TC_4$, $C_3^-$, and $C_3^+$
depicted in~\cref{fig:four-vertices}.

\begin{figure}[ht!]
\centering
\begin{tikzpicture}

  \begin{scope}[xshift = -4.5cm, scale = 0.8]
    \node [vertex] (1) at (-1,2) {};
    \node [vertex] (2) at (1,2) {};
    \node [vertex] (4) at (-1,0) {};
    \node [vertex] (3) at (1,0) {};
    \node (L1) at (0,-1) {$T_4$};
      
    \foreach \from/\to in {1/2, 1/3, 1/4, 2/3, 2/4, 3/4}     
    \draw [arc] (\from) to (\to);
  \end{scope}

  \begin{scope}[xshift = -1.5cm, scale = 0.8]
    \node [vertex] (1) at (-1,2) {};
    \node [vertex] (2) at (1,2) {};
    \node [vertex] (4) at (-1,0) {};
    \node [vertex] (3) at (1,0) {};
    \node (L1) at (0,-1) {$TC_4$};
      
    \foreach \from/\to in {1/2, 1/3, 4/1, 2/3, 2/4, 3/4}     
    \draw [arc] (\from) to (\to);
  \end{scope}

  \begin{scope}[xshift = 1.5cm, scale = 0.8]
    \node [vertex] (1) at (-1,2) {};
    \node [vertex] (2) at (1,2) {};
    \node [vertex] (4) at (-1,0) {};
    \node [vertex] (3) at (1,0) {};
    \node (L1) at (0,-1) {$C_3^-$};
      
    \foreach \from/\to in {1/2, 1/3, 1/4, 2/3, 4/2, 3/4}     
    \draw [arc] (\from) to (\to);
  \end{scope}

  \begin{scope}[xshift = 4.5cm, scale = 0.8]
    \node [vertex] (1) at (-1,2) {};
    \node [vertex] (2) at (1,2) {};
    \node [vertex] (4) at (-1,0) {};
    \node [vertex] (3) at (1,0) {};
    \node (L1) at (0,-1) {$C_3^+$};
      
    \foreach \from/\to in {2/1/, 3/1, 4/1, 2/3, 4/2, 3/4}     
    \draw [arc] (\from) to (\to);
  \end{scope}

\end{tikzpicture}
\caption{The four non-isomorphic oriented tournaments on $4$ vertices}
\label{fig:four-vertices}
\end{figure}
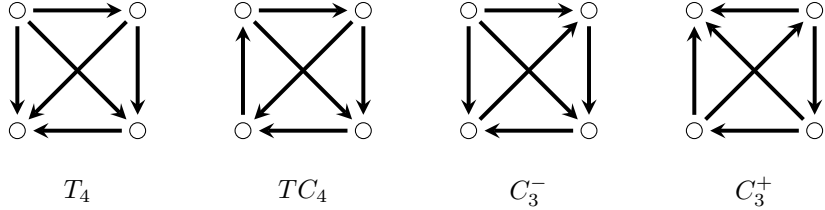

If $D$ is a digraph, then we denote by $-D$ the digraph obtained by flipping the
orientation of all arcs of $D$, and we call it the \textit{flip} of $D$. It is
not hard to notice that if $D$ is a $T_k$-free digraph,  then $-D$ is also $T_k$-free. 
The following example shows that there are finite sets of tournaments
that are not preserved by flips. 

\begin{example}
Note that $-(C_3^-)$ is isomorphic to $C_3^+$.
Hence, $C_3^-$ is a $\{T_4, C_3^+\}$-free oriented graph whose flip is not 
$\{T_4, C_3^+\}$-free.
\end{example} 

 There is a second source of possible symmetries between 
$\mathcal{F}$-free oriented graphs. Given  a vertex $a\in V(D)$, the \emph{switch of
$D$ with respect to $a$} is the oriented graph obtained by switching the orientation
of all arcs incident to $a$. We denote the resulting  oriented graph by $sw_a(D)$. 

\begin{example}
The digraph \red{$TC_4$} is the switch of $T_4$ with respect
$v$, where $v$ is any vertex of $T_4$ with positive in- and out-degree. 
\end{example} 

We say that a class of digraphs $\mathcal{C}$ is \emph{preserved by switches} (resp., flips) if for every $D\in \mathcal{C}$
and each $a\in V(D)$ it is the case that $\mathcal{C}$ contains a digraph isomorphic to  $sw_a(D)$ (resp., $-D$). 
It is evident that the class of $\mathcal{F}$-free digraphs is preserved by switches (resp., flips)
if and only if $\mathcal{F}$ is preserved by switches (resp., flips). 
For instance, the previous example shows that the class of $T_4$-free graphs is preserved by flips but not by switches.

Since $H_\mathcal F$ is the underlying graph of $D_\mathcal F$, every automorphism of
$D_\mathcal{F}$ is an automorphism of $H_\mathcal{F}$. 
Agarwal and Kompatscher~\cite{AgarwalKompatscher} gave a description of 
the possible automorphism groups of $H_{\mathcal F}$ in terms of the automorphism group
of $D_{\mathcal F}$, and a few additional permutations that we introduce in the following paragraphs.


For any set $X$, 
we denote by $\Sym(X)$ the permutation group of all \red{permutations of $X$}. If
$-D_\mathcal{F} \cong D_\mathcal{F}$, we denote by $\mi$ any isomorphism
$\mi \colon D_\mathcal{F}\to -D_\mathcal{F}$. In other words, $\mi
\colon V(D_\mathcal{F})\to V(D_\mathcal{F})$ is a bijection such that $(x,y)\in 
E(D_\mathcal{F})$ if and only if $(\mi(y),\mi(x))\in E(D_\mathcal{F})$.
Similarly, if $sw_a(D_\mathcal{F})\cong D_\mathcal{F}$ for some $a\in V(D_\mathcal{F})$,
we denote by $\sw$ any isomorphism $\sw \colon D_\mathcal{F}\to sw_a(D_\mathcal{F})$. That is, 
$\sw \colon V(D_\mathcal{F})\to V(D_\mathcal{F})$ is a bijection such that $(a,x)\in 
E(D_\mathcal{F})$ if and only if $(\sw(x),\sw(a))\in E(D_\mathcal{F})$, 
$(x,a)\in  E(D_\mathcal{F})$ if and only if $(\sw(a),\sw(x))\in E(D_\mathcal{F})$, 
and whenever $x\neq a\neq y$, there is an edge $(x,y)\in  E(D_\mathcal{F})$ if and
only if $(\sw(x),\sw(y))\in E(D_\mathcal{F})$. 

\begin{lemma}\label{lem:flip}
 The following statements are equivalent for a finite set of finite tournaments $\mathcal{F}$.
 \begin{enumerate}
     \item $\mathcal{F}$ is preserved by flips.
     \item The class of $\mathcal{F}$-free digraphs is preserved by flips. 
     \item $D_\mathcal{F}$ is isomorphic to $-D_\mathcal{F}$.
     \item $\mi \colon V(D_\mathcal{F})\to V(D_\mathcal{F})$ exists and $\mi \in
     \Aut(H_\mathcal{F})$.
 \end{enumerate}
\end{lemma}
\begin{proof}
The first two items are clearly equivalent. 
It is immediate to notice that if $\mi \colon V(D_\mathcal{F})\to
V(D_\mathcal{F})$ exists, then $\mi$ defines an automorphism of $H_\mathcal{F}$.
By the definition of $\mi \colon
V(D_\mathcal{F})\to V(D_\mathcal{F})$, the last two items are clearly equivalent. 
\red{Since a finite digraph embeds into $D_{\mathcal F})$ if and only if it is ${\mathcal F}$-free, if $D_{\mathcal F}$ is isomorphic to 
$-D_{\mathcal F}$, then the class of ${\mathcal F}$-free graphs is closed under flips, so item three implies item two. 
Conversely, if the class of finite digraphs that embeds into $D_{\mathcal F}$ is closed under flips, then by the homogeneity of $D_{\mathcal F}$ we can construct an isomorphism between $D_{\mathcal F}$ and $-D_{\mathcal F}$ inductively.} 
\end{proof}

\begin{lemma}\label{lem:switch}
 The following statements are equivalent for a finite set of finite tournaments $\mathcal{F}$.
 \begin{enumerate}
     \item $\mathcal{F}$ is preserved by switches.
     \item The class of $\mathcal{F}$-free digraphs is preserved by switches.
     \item $D_\mathcal{F}$ is isomorphic to $sw_a(D_\mathcal{F})$ for some $a\in
     V(D_\mathcal{F})$.
     \item $D_\mathcal{F}$ is isomorphic to $sw_a(D_\mathcal{F})$ for each $a\in
     V(D_\mathcal{F})$.
     \item $\sw \colon V(D_\mathcal{F})\to V(D_\mathcal{F})$ exists and $\sw \in
     \Aut(H_\mathcal{F})$.
 \end{enumerate}
\end{lemma}
\begin{proof}
The first two statements are clearly equivalent, and the last two are equivalent by the definition of $\sw$.
\red{Now suppose that the class of ${\mathcal F}$-free digraphs is preserved by switches. 
Then for any $a \in V(D_{\mathcal F})$ an isomorphism between $D_{\mathcal F}$ and $\sw_a(D_{\mathcal F})$ can be constructed inductively by the homogeneity of $D_{\mathcal F}$.
Hence, 2. implies 4. Trivially, 4. implies 3., and 3. implies 2 follows from the definitions.} 
\end{proof}

As anticipated, the following 
 statement asserts that $\Aut(H_\mathcal{F})$ is either $\Aut(D_\mathcal{F})$ or
the smallest closed supergroup of  $\Aut(D_\mathcal{F})$ that contains $\mi$, $\sw$, or both. 
The following statement is an adaptation of the third statement of Theorem 2.2 in~\cite{AgarwalKompatscher}
to the notation of the present work.

\begin{theorem}[Theorem 2.2(iii) in \cite{AgarwalKompatscher}]
\label{thm:AK}
    Let $\mathcal F$ be a finite set of finite tournaments. 
    Then, $H_\mathcal F$ is the Rado graph, $H_\mathcal F$ is a Henson graph, or $H_\mathcal F$ is not homogeneous. In the last case, the automorphism group of $H_{\mathcal F}$ equals one of the permutation groups from the following list.
    \begin{enumerate} 
        \item $\Aut(D_{\mathcal F})$.
        \item $\langle \Aut(D_{\mathcal F}) \cup \{\mi\} \rangle$.
        \item $\langle \Aut(D_{\mathcal F}) \cup \{\sw\} \rangle$.
        \item $\langle \Aut(D_{\mathcal F}) \cup \{\mi,\sw\} \rangle$.
    \end{enumerate}  
\end{theorem}

Our approach for proving dichotomy for the $\mathcal{F}$-free orientation problem
will follow a case distinction on $\Aut(H_\mathcal{F})$ according to the
cases listed above. For this, it will be convenient to describe for which
sets $\mathcal F$ the graph $H_\mathcal F$ is the Rado graph or a Henson graph.

\begin{lemma}\label{lem:Rado-Henson}
The following statements hold for a finite set of finite tournaments $\mathcal F$.
\begin{enumerate}
    \item \red{$\Csp(H_\mathcal F)$ is the class of all loopless graphs} if and only if $\mathcal F$ contains no transitive tournament.
    \item \red{$\Csp(H_\mathcal F)$ is the class of $K_n$-free graphs}  if and only if 
    \begin{itemize}
        \item there is no $\mathcal F$-free tournament with $n$ vertices,
        \item $\mathcal F$ contains a transitive tournament, and 
        \item $n_\mathcal F = n$. 
        \end{itemize}
\end{enumerate}
\end{lemma}
\begin{proof}
    If $\mathcal F$ contains no transitive tournament, then \red{every graph $G$} admits an $\mathcal F$-free orientation
    (orient the edges of \red{$G$} according to any linear ordering of \red{$V(G)$}), and so it follows that
    \red{$\Csp(H_\mathcal F)$} is the \red{class of all loopless graphs}. 
    If $\mathcal F$ contains a transitive tournament, then it follows from
    \cite{erdosMTA9} that there is a complete graph that does not admit an $\mathcal F$-free orientation, and
    thus \red{$\Csp(H_\mathcal F)$ does not contain all loopless graphs.}
    
     With similar arguments as in the \red{first itemized statement}, one can notice that if $\mathcal F$ contains a transitive
    tournament and there is no $\mathcal F$-free \red{tournament on $n_\mathcal F$ vertices}, then \red{$\Csp(H_\mathcal F)$ is the 
    class of $K_n$-free graphs for $n = n_\mathcal F$}.
    \red{Conversely, if $\Csp(H_\mathcal F)$ is the class of $K_n$-free graphs, then there is no $\mathcal F$-free 
    tournament on $n$ vertices. In particular, $\mathcal F$ contains a transitive tournament on at most $n$ vertices, i.e., 
    $n_\mathcal F \le n$. All that is left to prove is that $n_\mathcal F \ge n$. Proceeding by contradiction, suppose that 
    $n_\mathcal F  < n$.} It suffices to show that there is a graph $G$ with no complete subgraph
    on $n_\mathcal F +1$ vertices such that $G$ does not admit an $\mathcal F$-free orientation, i.e., \red{$\Csp(H_\mathcal F)$
    does not contain all $K_n$-free graphs (a contradiction).}
    It is well-known that for every positive integer $k\ge 3$, there is a $K_k$-free graph $G$ such that
    every $2$-edge-colouring of $G$ yields a monochromatic complete graph on $k-1$ vertices (see,  e.g.,~\cite{folkmanSIAM18}). 
    Let $G$ be such a graph for $k = n_\mathcal F +1$. Consider any orientation $G'$ of $G$, and
    any linear ordering $\le$ of $V(G)$. Now, colour an edge $xy$ of $G$ with blue if $x\le y$
    and $(x,y)\in E(G')$ or if $y \le x$ and $(y,x)\in E(G')$; otherwise, colour $xy$ with red. 
    By the choice of $G$, there must \red{be} $k$ vertices $v_1,\dots, v_k$ that induce a monochromatic
    clique. It is not hard to notice that $v_1,\dots, v_k$ must induce a transitive tournament on
    $G'$, so $G'$ contains a transitive tournament on \red{$n_\mathcal F$} vertices. Since this holds
    for any orientation $G'$ of $G$, we conclude that $G$ does not admit an $\mathcal F$-free orientation.
    Thus, $G$ is \red{a $K_n$-free graph that does not belong to $\Csp(H_\mathcal F)$.}
    This concludes the proof. 
\end{proof}

Finally, it will also be convenient to have an alternative description
of $\Aut(H_\mathcal{F})$ if it contains the action $\sw$. Let $P \subseteq V^3$
be the ternary relation that contains all triples $(i,j,k) \in V^3$ such that 
$U(i,j)$, $U(j,k)$, $U(i,k)$, and $|\{(i,j),(j,k),(i,k)\} \cap E| \text{ is even}$. 

\begin{theorem}\label{thm:P} 
    For a finite set of finite tournaments $\mathcal{F}$ the following equalities hold.
    \begin{enumerate}
        \item $ \langle \Aut(D_{\mathcal F}) \cup \{\sw\} \rangle = \Aut(V;P)$, and
        \item $\langle \Aut(D_{\mathcal F}) \cup \{-,\sw\} \rangle = \langle \Aut(V;P)
    \cup \{-\} \rangle$.
    \end{enumerate}
\end{theorem}
\begin{proof}
First observe that $P$ is preserved by $\sw$; this implies the inclusions $\subseteq$ in the two statements. 
Now suppose for contradiction that there exists $\alpha \in \Aut(V;P) \setminus \langle \Aut(D_{\mathcal F}) \cup \{\sw\} \rangle$. 
Then Theorem~\ref{thm:AK} implies that $\Aut(H_{\mathcal F})$ contains $\mi$, which is a contradiction because $\mi$ clearly does not preserve $P$. 
\end{proof} 


\red{An $\omega$-categorical structure is \emph{model complete} if its automorphism group is dense in the self-embeddings; equivalently, if in the structure every first-order formula is equivalent to an existential formula (see, e.g.,~\cite{Book}).}  

\begin{theorem}\label{thm:mc-core}
For all finite sets ${\mathcal F}$ of finite tournaments, $H_{\mathcal F}$ is a model-complete core. 
\end{theorem}
\begin{proof}
    The fact that $H_\mathcal{F}$ is core follows from \cref{lem:HF-core}, and the
    model completeness of $H_\mathcal{F}$ is a side product of the proof of \red{Theorem 2.2} 
    in~\cite{AgarwalKompatscher} (\red{via \emph{canonical functions}; also see~\cite{RandomMinOps} for another use of this technique where the extra consequences for model completeness are made explicit}). 
\end{proof}

\section{The Orientation Problem}
\label{sect:dicho-o}
We prove the complexity dichotomy for the $\mathcal F$-free orientation problem following
a similar idea as we proved the dichotomy for the $\mathcal F$-free orientation completion problem:
we show that for each finite set of finite tournaments $\mathcal F$, there is a Boolean structure
whose CSP is polynomial-time equivalent to the $\mathcal F$-free orientation problem. We will
do so by a case distinction over the automorphism group of $H_\mathcal{F}$, and repeatedly use the
following lemma (see~\cite{Book}).

\begin{lemma}\label{lem:mc-core-orbits}
If $\mathfrak C$ is an $\omega$-categorical model-complete core, then all orbits of $k$-tuples of $\Aut(\mathfrak C)$ are primitively positively definable in ${\mathfrak C}$. 
\end{lemma}

Throughout the remaining of this section, let ${\mathcal F}$ be a fixed finite set of finite tournaments.

\subsection{The Standard Case}

 We begin by considering the case where $\Aut(H_{\mathcal F}) = \Aut(D_{\mathcal F})$.   
     
    \begin{proposition}\label{prop:standard-def-s}
    Let $\mathcal F$ be a finite set of finite tournaments with  $n_{\mathcal F} \geq 3$. If $\Aut(H_{\mathcal F}) = \Aut(D_{\mathcal F})$, then the relation $E$ has a primitive positive definition in $H_{\mathcal F}$. 
    \end{proposition}
    \begin{proof}     
    Note that $E$ consists of one orbit of pairs in $\Aut(H_{\mathcal F})$. 
    Since $H_{\mathcal F}$ is a model-complete core by \cref{thm:mc-core}, it follows that the relation $E$ has a primitive positive definition in $H_{\mathcal F}$ by Lemma~\ref{lem:mc-core-orbits}. 
    \end{proof}

\begin{corollary}\label{cor:standard-pp-interpret}
    Let $\mathcal F$ be a finite set of finite tournaments with $n_{\mathcal F} \geq 3$. If $\Aut(H_{\mathcal F}) = \Aut(D_{\mathcal F})$, then 
    there exists a primitive positive interpretation of $\bB_{\mathcal F}$ in $H_{\mathcal F}$, and  there exists a polynomial-time reduction 
    from $\Csp(\bB_{\mathcal F})$ to the ${\mathcal F}$-free orientation problem. 
\end{corollary}
\begin{proof}
   Since $E$ has a primitive positive definition in $H_{\mathcal F}$ by~\cref{prop:standard-def-s}, the statement is an immediate consequence of~\cref{lem:interpret-B}. The polynomial-time reduction follows from~\cref{lem:pp-interpret-reduce}. 
\end{proof}


\subsection{The Flipping Case}

We write $N$ for the binary relation on $V$ defined as 
\[
N := \{(x,y) \in V^2 \mid \neg U(x,y) \wedge x \neq y\}.
\]
Let $O$ be the arity four relation on $V$  defined as
    \[
    O := \{(x,y,u,v) \in \red{V^4} \mid S_4(x,y,u,v) \wedge N(x,u) \wedge N(x,v) \wedge N(y,u) \wedge N(y,v) \}
    \]
    where $S_4$ is the relation defined in~\cref{eq:S4}. 
    
    \begin{lemma}\label{lem:flip-def-s}
        Let $\mathcal F$ be a finite set of finite tournaments with  $n_{\mathcal F} \geq 3$. If
        $\Aut(H_{\mathcal F}) = \langle \Aut(D_{\mathcal F}) \cup \{-\} \rangle$, then the relation $S_4$ 
        has a primitive positive definition in $H_{\mathcal F}$. 
    \end{lemma}
    \begin{proof}
     Note that the relation $O$ consists of one orbit of pairs in $\Aut(H_{\mathcal F})$. 
    Since $H_{\mathcal F}$ is a model-complete core by \cref{thm:mc-core}, the relation $O$ has a primitive positive definition in $H_{\mathcal F}$  by Lemma~\ref{lem:mc-core-orbits}.  
    We claim that $S_4$ has the following primitive positive definition.
    \begin{align*}
	\exists a,b \big (O(x,y,a,b) \wedge O(a,b,u,v) \big )
    \end{align*}
   	If $(x,y,u,v)$ satisfies this formula and $a,b$ are witnesses for the existentially quantified variables, then 
	 $(x,y) \in E$ if and only if
	 $(a,b) \in E$, which in turn is the case if and only if $(u,v) \in E$. Hence, the given formula implies $S_4(x,y,u,v)$. Conversely, if $(x,y) \in E$ and $(u,v) \in E$, then \red{by the homogeneity of $D_{\mathcal F}$} we may pick 
	$(a,b) \in E$ such that
	$(p,q) \in N$ for all $p \in \{x,y,u,v\}$ and $q \in \{a,b\}$, and hence $(x,y,u,v)$ satisfies the given formula.   
    \end{proof}

\begin{corollary}\label{cor:flip-reduction}
    Let $\mathcal F$ be a finite set of finite tournaments with $n_{\mathcal F} \geq 3$.
    If $\Aut(H_{\mathcal F}) = \langle \Aut(D_{\mathcal F}) \cup \{-\} \rangle$, then $\bB_{\mathcal F}$
    has a primitive positive interpretation in $H_{\mathcal F}$,
    and   there exists a polynomial-time reduction from 
    $\Csp(\bB_{\mathcal F})$ to 
    the ${\mathcal F}$-free orientation problem. 
\end{corollary}
\begin{proof}
   The first statement follows directly from~\cref{lem:flip-def-s} via~\cref{lem:interpret-B}. 
    The second statement then follows via~\cref{lem:pp-interpret-reduce}. 
\end{proof}

\subsection{The Switching Case}
So far, we have solely worked with the Boolean structure $\bB_\mathcal F$.
Now, when $\Aut(H_{\mathcal F})$ contains $\sw$, 
we consider an auxiliary 
Boolean structure which we define in the following paragraphs.
If $T$ is a tournament with vertex set $\{1,\dots,n\}$, then $c_T
\in \{0,1\}^{k \choose 3}$
is defined as follows. The entries of $c_T$ will be indexed by 3-element subsets $\{i,j,k\}$
of $\{1,\dots,n\}$, written as $(c_T)_{ijk}$.  For all $\{i,j,k\} \in {V(T) \choose 3}$ with
$i < j < k$ we have that $(c_T)_{ijk} = 0$ if and only if $|\{(i,j),(j,k),(i,k)\} \cap E(T)|$
is even. Equivalently, $c_T$ is defined by the following equation over $\mathbb Z_2$
\begin{align}
(c_T)_{ijk} = (b_T)_{ij} + (b_T)_{ik} + (b_T)_{jk}.
\label{eq:cT}
\end{align}


Recall that the tuples $b_T$ fully \red{determine} the labeled tournament $T$.  Note that there can be different
tournaments $T$ and $T'$ such that $c_T = c_{T'}$. Nonetheless, we argue that the tuples $c_T$ determine
whether $T$ is $\mathcal F$-free whenever $\mathcal F$ is preserved by switch. 

\begin{observation}\label{obs:c_T}
Let $\mathcal F$ be a finite set of tournaments preserved by switch, and $T, T'$ a pair of tournaments.
If $c_T = c_{T'}$, then $T$ is ${\mathcal F}$-free if and only if $T'$ is ${\mathcal F}$-free.
\end{observation}
\begin{proof}
This is a consequence of~\cref{lem:switch,thm:P}.
\end{proof}


\begin{definition}
The structure $\bC_{\mathcal F}$ has domain $\{0,1\}$ and the signature which contains for every $k \leq m_{\mathcal F}$ the 
relation symbol $Q_k$ of arity ${k \choose 3}$ which denotes in $\bC_{\mathcal F}$ the relation consisting of all $c \in  \{0,1\}^{k \choose 3}$ such that
there exists an $\mathcal F$-free tournament $T$ with $c_T = c$.
\end{definition}

We proceed to observe that \red{if $\sw \in \Aut(H_{\mathcal F})$, then} $\bB_\mathcal F$ and $\bC_\mathcal F$ are mutually
primitively positively definable. To do so, we will use the following lemma which is similar to~\cref{lem:minority-BF}.

\begin{lemma}\label{lem:minority-sw}
    Let $\mathcal F$ be a finite set of finite tournaments with 
    $\sw \in \Aut(H_{\mathcal F})$. If
    $\bC_{\mathcal F}$ does not interpret $K_3$ primitively positively,
    then $\bC_{\mathcal F}$ is preserved by the Boolean minority operation or a constant operation. 
\end{lemma}
\begin{proof}
If $\bC_{\mathcal F}$ is not NP-hard, it falls into one of the Schaefer's cases \red{(Theorem~\ref{thm:schaefer})}:
$\bC_\mathcal F$ is preserved by a constant operation, by $\min$, by $\max$, by $\majority$, or by $\minority$. If the first
or last case holds, then there is nothing to be shown. 
We show that if $\bC_\mathcal F$ is preserved by $\min$, by $\max$, or by $\majority$, then 
$\bC_{\mathcal F}$ is also preserved by the constant $0$ operation. 
It suffices to show that for every $n \leq m_{\mathcal F}$, the relation $Q_n$ is either empty or contains
the tuple with all entries $0$. For each $t \in Q_n$, define a tournament $T(t)$ on $\{1,\dots,m\}$ by setting
$(i,j) \in E(T(t))$ if $(t_i,t_j,t_{m+1}) \in P$, and setting $(j,i) \in E(T(t))$ if $(t_i,t_j,t_{m+1}) \not\in P$.
Denote by $\mathcal T$ the set of all tournaments on $\{1,\dots,m\}$ obtained in this way, and notice that
$P_{n-1}(\bB_\mathcal T)$ is preserved by min, max, or majority, because $Q_n$ is preserved by min, max, or majority. 
Thus, by \cref{lem:collapse}, either $P_{n-1}(\bB_\mathcal T) = \varnothing$ or $P_{n-1}(\bB_\mathcal T) = \{0,1\}^{n -1 \choose 2}$.
If the former case holds, then $Q_n$ is empty and the claim follows; if the latter holds, then
$P_{n-1}(\bB_\mathcal T)$ contains both constant tuples, i.e., $T_{n-1}\in \mathcal T$. 
Let $t \in Q_n$ be such that $T(t) = T_{n-1}$, i.e., $t_{ijn} = 0$ for all $1\le i < j \le n-1$. Since
$t = c_T$ for some $\mathcal F$-free tournament $T$, we have that $t_{ijk} = t_{ijn}  + t_{ikn} + t_{jkn} = 0$
for all $1\le i \le j \le k \le n-1$. Hence, $t$ is the constant $0$-tuple, and $t\in Q_n$. The claim now follows.
\end{proof}

Let $R_4 \subseteq \{0,1\}^4$
be the $4$-ary Boolean relation consisting of all the  tuples  $(i,j,k,l)$ such that $i + j + k + l = 0 \mod 2$.

\begin{lemma}\label{lem:BF-CF-T4}
Let $\mathcal F$ be a finite set of finite tournaments. If $\mathcal F$ is preserved by switch, 
then each pair of the following structures are mutually primitively positively definable
\[
    \bB_\mathcal F,~(\bB_\mathcal F,R_4),~(\bC_\mathcal F, R_4),~\bC_\mathcal F.
\]
\end{lemma}

\begin{proof}
We first see that $\bB_\mathcal F$ and $(\bB_\mathcal F, R_4)$ are mutually
primitively positively definable, and we only argue the non-trivial direction. By~\cref{thm:ppdef-pol},
it suffices to show that $R_4$ is preserved by all polymorphisms of $\bB_\mathcal F$. 
By \cref{lem:minority-BF}, if $\bB_\mathcal F$ has a polymorphism $f$ which is not a projection, then
it is generated by the minority, or by constant operations, i.e., $f$ is obtained by composing
projections with the minority or the constant operation (see e.g.,~\cite{post}).
Since $R_4$ is preserved by projections, by the minority, and by the constant operation, 
it follows inductively that $R_4$ is preserved by compositions of these operations. 
Therefore, $R_4$ is preserved by all polymorphisms of $\bB_\mathcal F$ and thus, it has a primitive positive definition in
$\bB_\mathcal F$. 
With similar arguments, and using \cref{lem:minority-sw}, one can prove that
$\bC_\mathcal F$ and $(\bC_\mathcal F,  R_4)$ are mutually primitively positively definable. 
To conclude the proof we show that $(\bB_\mathcal F,R_4)$ and
$(\bC_\mathcal F, R_4)$ primitively positively define one another.
It follows immediately from \cref{eq:cT}
that $Q_n^{\bB_\mathcal F}$ is
primitively positively defined in $(\bB_\mathcal F, R_4)$ by
\[
    \exists b_{12},\dots, b_{n-1,n} \left (P_n(b_{12},\dots, b_{n-1,n})
    \land\bigwedge_{1\le i < j < k\le n}R_4(b_{ij},b_{il},b_{jk}, c_{ijk}) \right).
\]
Finally, consider the ${k \choose 2}$-ary relation $P'_n(b_{12},\dots, b_{n-1,n})$ 
defined by the formula
\[
    \exists c_{123},\dots, c_{n-2,n-1,n} \left (Q_n(c_{123},\dots, c_{n-2,n-1,n})
    \land \bigwedge_{1\le i < j < k\le n}R_4(b_{ij},b_{ik},b_{jk}, c_{ijk}) \right).
\]
Notice that, for a tournament $T$ the formula $b_T \in P'_n$ 
if and only if there is an $\mathcal F$-free tournament $T'$ such 
that $c_T = c_{T'}$. Since $\sw\in \Aut(H_\mathcal F)$, by
\cref{obs:c_T} we conclude that $T$ is also $\mathcal F$-free, and
so $b_T\in P_n^{\bB_\mathcal F}$. Conversely, 
if $b_T\in P_n^{\bB_\mathcal F}$, then the tuple
$c_T$ is a witness to the fact that $b_T \in P'_n$.
\end{proof}

The following statement is an immediate implication of this lemma, \cref{lem:pp-interpret-reduce}, and \cref{thm:reduction-to-boolean}.

\begin{proposition}\label{prop:F-to-CF}
    Let $\mathcal F$ be a finite set of finite tournaments. If $\sw\in \Aut(H_\mathcal F)$, 
    then there is a  polynomial-time reduction from the $\mathcal F$-free orientation problem to 
    $\Csp(\bC_\mathcal F)$.
\end{proposition}

The proof of the following theorem is similar to the proof of Lemma~\ref{lem:pp-definition-S}. 
We use the ternary relation $P$ introduced before \cref{thm:P} to now define
\begin{align*}
    S_6 := \big \{(a,b,c,u,v,w) \in V^6 \mid \; & U(a,b) \wedge U(b,c) \wedge U(a,c) \wedge U(u,v) \wedge U(v,w) \wedge U(u,w) \\
    & \text{ and } P(a,b,c) \Leftrightarrow P(u,v,w) \big \} .
\end{align*}

\begin{lemma}\label{lem:S6-interpret}
For every finite set of finite tournaments $\mathcal F$, the Boolean structure
${\mathfrak C}_{\mathcal F}$ has a primitive positive interpretation in $(H_{\mathcal F}, S_6)$. 
\end{lemma}
\begin{proof}
Our interpretation $I$ has dimension three, is defined on all triples $(a,b,c)$ that induce a $K_3$ in $H_{\mathcal F}$, and is given by $I(a,b,c) := 0$ if $(a,b,c) \in P$, and $I(a,b,c) := 1$ otherwise. The domain formula $\top_I(a,b,c)$ is $U(a,b) \wedge U(b,c) \wedge U(a,c)$. The interpreting formula ${=_I}(a,b,c,u,v,w)$ for equality is $S_6(a,b,c,u,v,w)$. For $n \leq k_{\mathcal F}$, the interpreting formula $$(R_n)_I(a_{123},b_{123},c_{123},\dots,a_{n-2,n-1,n},b_{n-2,n-1,n},c_{n-2,n-1,n})$$ 
for $R_n$ is 
$$\exists x_1,\dots,x_n \bigwedge_{\{i,j\} \in {\red{[n]} \choose 2}} U(x_i,x_j) \wedge \bigwedge_{ \{i,j,k\} \in {\red{[n]} \choose 3} } S_6(a_{ijk},b_{ijk},c_{ijk},x_i,x_j,x_k).$$
It is straightforward to verify that $S_6(a,b,c,u,v,w)$ if and only if $I(a,b,c) = I(u,v,w)$. 
Now suppose that $(H_{\mathcal F},S_6) \models (R_n)_I(a_{123},\dots,c_{n-2,n-1,n})$, and let $d_1,\dots,d_n$ be witnesses for the existentially quantified variables in this formula. 
Note that $d_1,\dots,d_n$ must induce a $K_n$ in $H_{\mathcal F}$; let $T$ be the tournament induced in $D_{\mathcal F}$. Then $c_T \in R_n$ by the definition of $R_n$, which means that 
$$(I(a_{123},b_{123},c_{123}),\dots,I(a_{n-2,n-1,n},b_{n-2,n-1,n},c_{n-2,n-1,n})) \in R_n$$ by the properties of $S_6$. All of the implications in this argument can be \red{reversed}, which completes the proof. 
\end{proof}

To conclude the proof, it now suffices to show that $S_6$ has a primitive positive definition
in $H_\mathcal F$ when $\sw\in Aut(H_\mathcal F)$. Again, we consider the cases $-\in \Aut(H_\mathcal F)$
and $-\not\in\Aut(H_\mathcal F)$ separately. 

\begin{lemma}\label{lem:sw-fl}
    Let $\mathcal F$ be a finite set of finite tournaments. If $\Aut(H_{\mathcal F}) =
    \langle \Aut(D_{\mathcal F}) \cup \{-,\sw\} \rangle$,  then there exists a primitive positive definition of $S_6$ in $H_{\mathcal F}$.  
\end{lemma}

\begin{proof}
Notice that if $n_\mathcal F =3$, then $T_3\in \mathcal F$, and since $\overrightarrow{C_3}$
can be obtained as a switch of $T_3$, the triangle does not admit an $\mathcal F$-free
orientation. The latter claims holds trivially if $n_\mathcal F \le 2$. So, if $n_\mathcal F\le 3$,
then $S_6 = \varnothing$, and thus it is primitive positive defined by  formula $\bot$. Otherwise, if $n_\mathcal F\ge 4$,
we proceed similarly to~\cref{lem:flip-def-s}, but instead of $O$ using the relation 
    \begin{align*}
    (P(x_1,x_2,x_3) \Leftrightarrow (P(y_1,y_2,y_3)) & \wedge \bigwedge_{i,j\in \{1,2,3\}} N(x_i,y_j) \\
    & \wedge U(x_1,x_2) \wedge U(x_2,x_3) \wedge U(x_1,x_3) \\
    & \wedge U(y_1,y_2) \wedge U(y_2,y_3) \wedge U(y_1,y_3) 
    \end{align*} which consists of just one orbit of 6-tuples in 
    the group $\Aut(H_{\mathcal F})$, and hence has a primitive positive definition in the model-complete core structure $H_{\mathcal F}$. 
\end{proof}

Finally, we consider the most technical case, namely, when $\Aut(H_{\mathcal F}) = \langle \Aut(D_{\mathcal F}) \cup \{\sw\} \rangle$.

\begin{lemma}\label{thm:sw} 
    Let $\mathcal F$ be a finite set of finite tournaments.
    If $\Aut(H_{\mathcal F}) = \langle \Aut(D_{\mathcal F}) \cup
    \{\sw\} \rangle$, then $S_6$ has a  primitive positive definition 
    in $H_{\mathcal F}$. 
\end{lemma}
\begin{proof}
 With the same arguments as in the proof of \cref{lem:sw-fl}, we see that
if $n_\mathcal F\le 3$, then $S_6 $ is primitive positive defined by  formula
$\bot$. Suppose that  $n_\mathcal F\ge 4$, and notice that in this case, any orientation
of $K_3$ is $\mathcal F$-free, \red{a} fact which we use below. 
Now, since $\Aut(H_{\mathcal F}) = \langle \Aut(D_{\mathcal F}) \cup\{\sw\} \rangle$, it follows from Theorem~\ref{thm:P}
that the relation $P$ consists of one orbit of triples of $\Aut(H_{\mathcal F})$. Hence, $P$ has
a primitive positive definition in $H_{\mathcal F}$, because $H_{\mathcal F}$ is a model-complete core and by
Lemma~\ref{lem:mc-core-orbits}. 
The same holds for 
\begin{align*}
    Q & := \big \{(a,b,c) \in V^3 \mid U(a,b) \wedge U(b,c) \wedge U(a,c) \wedge \neg P(a,b,c) 
    \big \} 
\end{align*}
instead of $P$. 
Therefore, it suffices to prove that  $S_6$ 
is primitive positive definable in $(V;P,Q)$. 

Let $T \in {\mathcal F}$, 
and $\phi$ be a conjunction 
of all atomic $\{P,Q\}$-formulas 
of the form $P(i,j,k)$ for $\{i,j,k\} \in {\{1,\dots,k\} \choose 3}$ such that $(c_T)_{ijk} = 0$,
and of the form $Q(i,j,k)$ such that $(c_T)_{ijk} = 1$.  
Clearly, $\phi$ is unsatisfiable in $(V;P,Q)$. 
Let $\psi$ be a maximal satisfiable subset of the conjuncts of $\phi$.
Since $T_3$ and $\overrightarrow{C_3}$ are $\mathcal F$-free, each conjunct of $\phi$ is satisfiable and so 
$\psi$ has at least one conjunct $\chi_1$. Let $\psi'$ be the remaining conjuncts of $\psi$, i.e., $\psi' := \psi \setminus \{\chi_1\}$. 
Let $\chi_2$ be any conjunct of $\phi$ which is not in $\psi$. Note that $\psi'$  implies $\chi_1 \Rightarrow \neg \chi_2$. We may assume that $\chi_1$ is of the form $P(i,j,k)$ and $\chi_2$ is of the form $Q(i,j,k)$. 
To see that this is without loss of generality, first suppose that 
otherwise all conjuncts in $\phi\setminus \psi$ are of the form $P(i,j,k)$. 
Clearly, we may assume that 
two of the variables $a,b \in \{1,\dots,k\}$ 
  of $\chi_2$
  are smallest among all the variables of $\psi$. 
Let $\tilde T$ be an isomorphic copy obtained from $T$ by exchanging the labels $a$ and $b$; then the construction of $\phi$ above, carried out with $\tilde T$ instead of $T$, has the desired property. The case that all conjuncts of $\psi$ are of the form $Q(i,j,k)$ can be treated similarly. 

Note that $\chi_1$ and of $\chi_2$ cannot have the same set of variables;
so we may assume without loss of generality that $\chi_1 \Rightarrow \chi_2$ is of the form
$P(a,b,c) \Rightarrow P(a',b',c')$ where 
$a$ might be the same variable as $a'$, and
$b$ might be the same variable as $b'$, but $c$ and $c'$ are different variables. 
Let $\psi''$ be the primitive positive formula obtained from $\psi'$ by existentially quantifying all variables except for $a,b,c,a',b',c'$, and let $\eta$
be the formula
\begin{align*}
    \exists a',b',c',u',v',w' \big ( \psi''(a,b,c,a',b',c') \wedge \psi''(b',c',a',v',w',u') \wedge \psi''(w',u',v',w,u,v)  \big ). 
\end{align*}
First note that if $(a,b,c) \in P$,
then the first conjunct of $\eta$ implies that $(a',b',c') \in P$, the second that
$(u',v',w') \in P$, and the third that 
$(u,v,w) \in P$.
 Conversely, 
$\eta(a,b,c,u,v,w)$ implies that if $(u,v,w) \in P$, then $(a,b,c) \in P$. 

We claim that the following formula $\nu(a,b,c,u,v,w)$ defines $S_6$: 
\begin{align*}
    \exists a',b',c',u',v',w' \big ( \eta(a,b,c,a',b',c') \wedge \eta(a',b',c',u',v',w') \wedge \eta(u',v',w',u,v,w)  \big ). 
\end{align*}
It is easy to see that $\nu(a,b,c,u,v,w)$ implies
$S_6(a,b,c,u,v,w)$. 

Now suppose that $(a,b,c,u,v,w) \in S_6$. 
First consider the case that $(a,b,c) \in P$ and $(u,v,w) \in P$. The case that $(a,b,c) \in Q$ and
$(u,v,w) \in Q$ can be treated similarly. 
Let $G$ be the canonical database of the primitive positive $\{E\}$-formula obtained from
replacing the $\{P,Q\}$-atoms of $\nu(a,b,c,u,v,w)$ by their primitive positive definition in $H_{\mathcal F}$.
Let $D$ be an ${\mathcal F}$-free orientation of $G$. 
We assume that the tournament $\{a,b,c\}$ induces the same tournament in $D_{\mathcal F}$ and in $D$, and that 
$\{u,v,w\}$ induces the same tournament in $D_{\mathcal F}$ and in $D$.
This is without loss of generality, because 
otherwise we may repeatedly compose $e$ with $\sw_a$ for $a \in \{a,b,c,u,v,w\}$ to obtain such an orientation. 
 
We now add directed edges to $D$ such that
$\{a,b,c,u,v,w\}$ induces the same digraph in $D_{\mathcal F}$ and in $D$. Note that the resulting digraph is still ${\mathcal F}$-free, 
because every tournament that embeds into $D$ must either embed into $\{a,b,c,u,v,w\}$, into
$\{a,b,c,a',b',c'\}$, into $\{a',b',c',u',v',w'\}$, or into $\{u',v',w',u,v,w\}$, because all edges of $D$ are  covered by one of these subsets. 
Hence, $D$ has an embedding $e$ into $D_{\mathcal F}$, and by homogeneity we may assume that $e(a) = a$, $e(b) = b$, and $e(c) = c$, $e(u) = u$, $e(v) = v$, and $e(w) = w$. This shows that $\nu(a,b,c,u,v,w)$ holds in $H_{\mathcal F}$. 
\end{proof}

\begin{corollary}\label{cor:sw-pp-interpret}
    Let $\mathcal F$ be a finite set of finite tournaments. 
    If $\Aut(H_\mathcal F)\in [\langle \Aut(D_{\mathcal F}) \cup \{\sw\} \rangle,$
    $\langle \Aut(D_{\mathcal F}) \cup \{-,\sw\} \rangle]$, then 
    there exists a primitive positive interpretation of $\bB_{\mathcal F}$ in $H_{\mathcal F}$,
    and  there exists a polynomial-time reduction  from $\Csp(\bB_{\mathcal F})$ to the ${\mathcal F}$-free orientation problem.
\end{corollary}
\begin{proof}
   By~\cref{thm:AK}, either $\Aut(H_{\mathcal F}) = \langle \Aut(D_{\mathcal F}) \cup \{-,\sw\} \rangle$
   or $\Aut(H_\mathcal F) = \langle \Aut(D_{\mathcal F}) \cup \{\sw\} \rangle$; the primitive positive definition
   of $S_6$ in $H_\mathcal F$ now follows by~\cref{lem:sw-fl,thm:sw}, respectively.
   Thus, by~\cref{lem:S6-interpret}, $\bC_\mathcal F$ has a primitive positive interpretation in $H_\mathcal F$ and so, 
   by~\cref{lem:BF-CF-T4} and composing interpretations, we conclude that $\bB_\mathcal F$ has a primitive positive
   interpretation in $H_\mathcal F$. 
   Finally, the polynomial time reduction follows from from~\cref{lem:pp-interpret-reduce}.
\end{proof}

\subsection{Complexity Classification}
Finally, we put all these cases together and obtain the following result. 

\begin{theorem}\label{thm:main}
    Let ${\mathcal F}$ be a finite set of finite tournaments. 
    Then exactly one of the following two cases applies.
    \begin{itemize}
        \item $K_3$ has a primitive positive  interpretation in ${\mathfrak B}_{\mathcal F}$ and in $H_{\mathcal F}$. In this case, $\Csp(H_{\mathcal F})$ and the
        ${\mathcal F}$-free orientation problem are NP-complete.  
        \item $\bB_\mathcal F$ has a constant operation or the minority operation as a polymorphism, and $H_{\mathcal F}$ \red{is homomorphically equivalent to a structure with} a ternary pseudo 
        \red{weak} near unanimity polymorphism.
        In this case, $\Csp(H_{\mathcal F})$ and the 
        ${\mathcal F}$-free orientation problem are in P. 
    \end{itemize}
\end{theorem}
\begin{proof}
It is well-known that the two cases are mutually disjoint: if \red{$K_3$ has a primitive positive interpretation in $H_{\mathcal F}$, and $H'$ is homomorphically equivalent to $H_{\mathcal F}$, then
$H'$ primitively positively interprets a structure which is homomorphically equivalent to 
$K_3$~\cite{wonderland}. 
But then, $H'$ cannot have a ternary pseudo weak near unanimity polymorphism (see, e.g., Corollary 10.3.5 in~\cite{Book}).}

We prove that one of the two cases applies by a case distinction on \red{$\Csp(H_\mathcal F)$} and on $\Aut(H_{\mathcal F})$ according to \cref{thm:AK}. 
If \red{$\Csp(H_\mathcal F$)} is the \red{class of all loopless graphs}  (in particular, if ${\mathcal F}$ is empty), then, by \cref{lem:Rado-Henson},
$\mathcal F$ contains no transitive tournament, so for all $n \in \{2,\dots, m_{\mathcal F}\}$ 
we have that
$\red{P_k}^{{\mathfrak B}_{\mathcal F}}$ 
contains both constant ${ \red{k} \choose 2}$-tuples and so, 
${\mathfrak B}_{\mathcal F}$ has a constant polymorphism. Moreover, \red{$H_{\mathcal F}$ is homomorphically equivalent to 
the Rado graph} (Example~\ref{expl:Rado}), which clearly has a pseudo near unanimity polymorphism (Example~\ref{expl:Rado-pwnu})
and the statement is trivial.


If \red{$\Csp(H_\mathcal F)$} is \red{the class of $K_n$-free graphs for some $n\ge 2$}, then, by \cref{lem:Rado-Henson}, $\mathcal F$ contains a
transitive tournament, there is no $\mathcal F$-free tournament on $n_\mathcal F$ vertices, and $\red{n} = n_\mathcal F$.
So, $\red{P_k}^{{\mathfrak B}_{\mathcal F}} = \emptyset$ for all $\red{k}\in \{n_\mathcal F, \dots, m_{\mathcal F}\}$, 
and $\red{P_k}^{{\mathfrak B}_{\mathcal F}}$ contains both constant ${ \red{k} \choose 2}$-tuples for all other $\red{k} < n_\mathcal F$. 
Hence, ${\mathfrak B}_{\mathcal F}$ again has a constant polymorphism. Moreover,  $H_{\mathcal F}$ \red{is homomorphically equivalent to} the Henson graph
$H_{n_{\mathcal F}}$ (Example~\ref{expl:Henson}), which
also has a pseudo near unanimity polymorphism as was mentioned earlier (Example~\ref{expl:Rado-pwnu}). 
The ${\mathcal F}$-free orientation problem is in P, because it suffices to check whether the give
input graph contains $K_{n_\mathcal F}$.

Otherwise, if $\Csp(H_\mathcal F)$ is \red{neither} the \red{class of all loopless graphs}  nor \red{the class of $K_n$-free graphs
for $n\ge 2$}, then \red{$H_\mathcal F$ is neither the Rado graph nor a Henson graph, and so} Theorem~\ref{thm:AK} implies that
$\Aut(H_\mathcal F) \in [\Aut(D_\mathcal F), \langle \Aut(D_\mathcal F) \cup\{-,\sw\}\rangle]$. \red{Also, }
Lemma~\ref{lem:Rado-Henson} implies that $\mathcal F$ contains a transitive tournament, and that there is at least one
$\mathcal F$-free oriented tournament on $n_\mathcal F$ vertices. In particular, this implies that $\bB_\mathcal F$ does not have a
constant polymorphism, since  $P_{n_\mathcal F}^{\bB_\mathcal F}$ is neither empty, nor contains constant tuples. 
Hence, by \cref{lem:minority-BF}, $\bB_\mathcal F$ either has the minority operation as polymorphism, or it primitively positively interprets
$K_3$. If the former holds, then the minority operation is a polymorphism of $(\bB_\mathcal F, \bf{0},\bf{1})$ as well,
and thus, $(D_\mathcal F,U)$ has a pseudo \red{weak} near unanimity polymorphism $f$ by \cref{thm:or-comp-classification}.
Clearly, $f$ is also a pseudo near  unanimity polymorphism of $H_\mathcal F$. The fact that the $\mathcal F$-free
orientation problem and $\Csp(H_\mathcal F)$ are in P follows from the polynomial-time tractability of
$\Csp(\bB_\mathcal F)$ and \cref{thm:reduction-to-boolean}.
Finally, suppose that $\bB_\mathcal F$ primitively positively interprets $K_3$. Since
$\Aut(H_\mathcal F) \in [\Aut(D_\mathcal F), \langle \Aut(D_\mathcal F) \cup\{-,\sw\}\rangle]$,
 then, by~\cref{thm:AK}, it must be the case that either $\Aut(H_\mathcal F) = \Aut(D_\mathcal F)$, 
 $\Aut(H_\mathcal F) =  \langle \Aut(D_\mathcal F) \cup\{-\}\rangle$, or $\Aut(H_\mathcal F)\in [\langle
\Aut(D_{\mathcal F}) \cup \{\sw\} \rangle, \langle \Aut(D_{\mathcal F}) \cup \{-,\sw\} \rangle]$. So, in each case
by Corollaries \ref{cor:standard-pp-interpret}, \ref{cor:flip-reduction}\red{,} and~\ref{cor:sw-pp-interpret},
respectively, there is a 
primitive positive interpretation of $\bB_\mathcal F$ in $H_\mathcal F$.  Finally, by the well-known
fact that primitive positive interpretations compose, we conclude that $K_3$ has a primitive positive interpretation in
$H_\mathcal F$, and $\Csp(H_\mathcal F)$ and the $\mathcal F$-free orientation problem are NP-complete by
\cref{lem:pp-interpret-reduce}.
 \end{proof}


\section{Examples and Applications}
\label{sect:examples}

If the reader is not familiar with constraint satisfaction theory, they
might find~\cref{thm:or-comp-classification,thm:main} not transparent enough. 
In order to address a broader audience, we first describe the minority
operation in terms of tournaments, and then, using this description, we
propose simpler versions of~\cref{thm:or-comp-classification,thm:main}. 
Some readers might find these versions more natural, and others might find
it redundant; the latter can skip the following subsection.

We conclude this section by providing some examples and applications
of \cref{thm:or-comp-classification,thm:main} to certain natural instances of
the $\mathcal F$-free orientation and orientation completion problems.  

\subsection{Minority and Tournaments}

Consider three tournaments $T^1$, $T^2$, and $T^3$ with vertex set $[n]$. 
The \textit{minority operation} maps the triple $(T^1,T^2,T^3)$ to a
tournament with vertex set $[n]$ which we denote by $\minority(T^1,T^2,T^3)$ and the edge
set is defined as follows. For $1\le i < j \le n$ there is an edge $(i,j)$ in
$\minority(T^1,T^2,T^3)$ if either $(i,j)\in E(T^1)\cap E(T^2)\cap E(T^3)$
or if $(i,j)$ is an edge in exactly one of the tournaments $T^1$, $T^2$, and $T^3$. 
Intuitively, when deciding whether there is an edge $(i,j)$ in  $\minority(T^1,T^2,T^3)$,
we first check if the three tournaments agree on $(i,j)$, and otherwise, 
we take the minority vote.
We extend these operation to take as input triples of tournaments $(T^1,T^2,T^3)$
with different sizes of vertex sets. Suppose that $V(T^1) = [n_1]$,  $V(T^2) = [n_2]$,
and $V(T^3) = [n_3]$, where  $n_1 \le n_2 \le n_3$. We define
$\minority(T^1,T^2,T^3): = \minority(T^1,T^2[n_1],T^3[n_1])$.

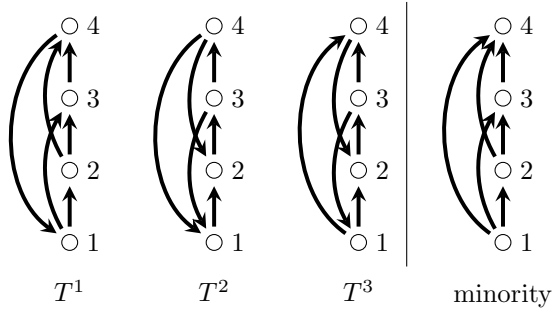
\begin{figure}[ht!]
\centering
\begin{tikzpicture}[scale = 0.8]

  \begin{scope}[xshift = 0cm, yshift = -6cm, scale = 0.8]
    \node [vertex, label=0:{$1$}] (1) at (-6,0) {};
    \node [vertex, label=0:{$2$}] (2) at (-6,1.5) {};
    \node [vertex, label=0:{$3$}] (3) at (-6,3) {};
    \node [vertex, label=0:{$4$}] (4) at (-6,4.5) {};
    \node (L1) at (-6,-1) {$T^1$};
      
     \foreach \from/\to in {1/2, 2/3, 3/4}     
    \draw [arc] (\from) to (\to);
    \draw [arc] (2) to [bend left] (4);
    \draw [arc] (1) to [bend left] (3);
    \draw [arc] (4) to [bend right = 55] (1);
  
    \node [vertex, label=0:{$1$}] (1) at (-3,0) {};
    \node [vertex, label=0:{$2$}] (2) at (-3,1.5) {};
    \node [vertex, label=0:{$3$}] (3) at (-3,3) {};
    \node [vertex, label=0:{$4$}] (4) at (-3,4.5) {};
    \node (L1) at (-3,-1) {$T^2$};
      
    \foreach \from/\to in {1/2, 2/3, 3/4}     
    \draw [arc] (\from) to (\to);
    \draw [arc] (4) to [bend right] (2);
    \draw [arc] (3) to [bend right] (1);
    \draw [arc] (4) to [bend right = 55] (1);

    \node [vertex, label=0:{$1$}] (1) at (0,0) {};
    \node [vertex, label=0:{$2$}] (2) at (0,1.5) {};
    \node [vertex, label=0:{$3$}] (3) at (0,3) {};
    \node [vertex, label=0:{$4$}] (4) at (0,4.5) {};
    \node (L1) at (0,-1) {$T^3$};
      
    \foreach \from/\to in {1/2, 2/3, 3/4}     
    \draw [arc] (\from) to (\to);
    \draw [arc] (4) to [bend right] (2);
    \draw [arc] (3) to [bend right] (1);
    \draw [arc] (1) to [bend left = 55] (4);

    \draw (1,-0.5) -- (1,5);

    \node [vertex, label=0:{$1$}] (1) at (3,0) {};
    \node [vertex, label=0:{$2$}] (2) at (3,1.5) {};
    \node [vertex, label=0:{$3$}] (3) at (3,3) {};
    \node [vertex, label=0:{$4$}] (4) at (3,4.5) {};
    \node (L1) at (3,-1.1) {$\minority$};
      
    \foreach \from/\to in {1/2, 2/3, 3/4}     
    \draw [arc] (\from) to (\to);
    \draw [arc] (2) to [bend left] (4);
    \draw [arc] (1) to [bend left] (3);
    \draw [arc] (1) to [bend left = 55] (4);
  \end{scope}
 
\end{tikzpicture}
\caption{An illustration of the minority operation acting on a triple $(T^1,T^2,T^3)$ of tournaments
isomorphic to $TC_4$, which yields a tournament isomorphic $T_4$.}
\label{fig:operations2}
\end{figure}

We say that a set of tournaments $\mathcal{T}$ is \textit{preserved by the minority operation}
if for every three tournaments $T^1,T^2,T^3$ with vertex set $[n_1]$, $[n_2]$, and $[n_3]$,
respectively, whenever each $T^i$, with $i\in\{1,2,3\}$, is isomorphic to some tournament in $\mathcal T$,
then  $\minority(T^1, T^2, T^3)$ is isomorphic to some tournament in $\mathcal{T}$.
For instance, the illustration of this operation in \cref{fig:operations2}
shows if a set $\mathcal T$ contains $TC_4$, but not $T_4$, then it is not preserved by the minority
operation.

Clearly, the previously defined  operation  on tournaments is a translation of
the minority operation on Boolean relational structures (Theorem~\ref{thm:schaefer}). 
We will use these translations to propose a classification of the complexity of the
$\mathcal{F}$-free orientation completion problem. Given a non-empty finite set of
finite tournaments $\mathcal{F}$, we denote by $\mathcal{F}_f$ the set of  $\mathcal{F}$-free
tournaments on at most $m_\mathcal{F}$ vertices.

\begin{lemma}\label{lem:operations}
    Let $\mathcal F$ be a finite set of finite tournaments. The Boolean structure
    $\bB_\mathcal{F}$ is preserved by the minority Boolean operation if
    and only if $\mathcal{F}_f$ is preserved by the minority operation.
\end{lemma}
\begin{proof}
By definition of $P_n(\bB_\mathcal F)$, it follows that for each $n\le m_\mathcal F$ the
relation $P_n$ is preserved by the minority operation if and only if for any three
$\mathcal F$-free tournaments $T^1, T^2, T^3$ with vertex set $[n]$, it is the case that
$\minority(T^1, T^2, T^3)$ is $\mathcal F$-free. Thus, it suffices to prove that the latter
condition holds if and only if $\mathcal{F}_f$ is preserved by the minority operation. One implication
is trivial, we prove the remaining one by contraposition. Suppose $\mathcal{F}_f$ is not 
preserved by minority, and let $T^1, T^2, T^3\in \mathcal{F}_f$ such that 
$\majority(T^1, T^2, T^3)\not\in \mathcal{F}_f$. Without loss of generality, assume that 
$T^1$ has the minimum number of vertices $n$ amongst these three tournaments. Since $\mathcal{F}_f$
is closed under taking subtournaments, $T^2[n]$ and $T^3[n]$ belong to 
$\mathcal{F}_f$. The claim follows since $\minority(T^1, T^2[n], T^3[n])$ is equal 
(by definition) to  $\minority(T^1, T^2, T^3)$ which does not belong to $\mathcal{F}_f$.
\end{proof}

Now, we translate \cref{thm:or-comp-classification,thm:main} in terms of the minority operation
and tournaments. 

\begin{corollary}\label{cor:or-completion-class}
    For every finite set of finite tournaments $\mathcal F$ one of the following cases holds.
    \begin{enumerate}
        \item $\mathcal F_f$ is preserved by the minority operation.  In this case, the
        \red{$\mathcal F$}-free orientation completions of a digraph $D$ correspond
        to the solution space of a system of linear equations over $\mathbb Z_2$
        (constructed from $D$ and $\mathcal F$). 
        \item Otherwise, $\mathcal F$-free orientation completion problem is NP-complete.
    \end{enumerate}
    In the first case, the $\mathcal F$-free orientation completion problem is in P.
\end{corollary}
\begin{proof}
    It follows directly from \cref{thm:or-comp-classification,lem:operations}.
\end{proof}

\begin{corollary}\label{cor:or-classification}
    For every finite set of finite tournaments $\mathcal F$ one of the following cases holds.
    \begin{enumerate}
        \item $\mathcal F$ contains no transitive tournament. In this case, every graph admits
        an $\mathcal F$-free orientation.
        \item $\mathcal F_f$ is preserved by the minority operation.  In this case, the
        \red{$\mathcal F$}-free orientations of a graph $G$ correspond
        to the solution space of a system of linear equations over $\mathbb Z_2$
        (constructed from $G$ and $\mathcal F$).
        \item Otherwise, the $\mathcal F$-free orientation completion problem is NP-complete.
    \end{enumerate}
    In cases 1 and 2, the $\mathcal F$-free orientation problem is in P.
\end{corollary}
\begin{proof}
    Immediate implication of \cref{thm:main,lem:operations}.
\end{proof}

We can already illustrate these corollaries using~\cref{fig:four-vertices}. 

\begin{example}\label{ex:T4-TC4}
Notice that by the minority operation depicted in~\cref{fig:operations2}, if $\mathcal F$
contains $T_4$ and $TC_4$ \red{red} $\mathcal F$-free, then $\mathcal F_f$ is not preserved by the 
minority operation. Thus, in any such case, the $\mathcal F$-free orientation
and the $\mathcal F$-free orientation completion problems are NP-complete.
\end{example}

\subsection{Sets with $T_3$ or $\overrightarrow{C_3}$}

Now, we consider the special cases when $\mathcal F$ contains some small
tournament. Specifically, when $\mathcal F$ contains at least one tournament
on at most $3$ vertices. The cases when $\mathcal F$ contains $T_1$ is trivial, 
and it is also not hard to settle the case when  $\mathcal F$ contains $T_2$.

Recall that given a digraph $D$, we write $U$ to \red{denote} the symmetric closure
of $E$. It is not hard to notice that an oriented graph $D$ can be described as a
$\mathbb{Z}_2$-colouring $c\colon U\to \mathbb{Z}_2$ such that for
every pair $ij\in U$ the equality $c(i,j) + c(j,i) = 1$ is satisfied. 
Equivalently, orientation completions  of a digraph $D$
are in one to one correspondence with solutions to the linear equation
$x_{ij} + x_{ji} = 1$ over $\mathbb{Z}_2$, where $i,j$ ranges over
adjacent vertices of $G$, and $x_{ij} = 1$ for each $(i,j)\in E$ such that
$(j,i)\not\in E$. Now, notice that every hamiltonian oriented
path of $\overrightarrow{C_3}$ has an even number of forward edges, but
there are hamiltonian oriented paths of $T_3$ with different parity. Thus,
the $T_3$-free orientation completions of $D$ are in one-to-one correspondence
to the solutions of the system  with variables $x_{ij}$ for
$ij\in U$ and linear equations 
\begin{align*}
x_{ij} & = 1 \text{ where } (i,j)\in E \text{ and } (j,i)\not\in E, \\
x_{ij} + x_{ji} & = 1 \text{ where } ij \in U, \\
 x_{ij} + x_{jk} & = 0 \text{ where } ij, jk, ik \in U.
 \end{align*}
 We denote by Sys$_3(D)$ the previously described system of linear equations. 
For the following statement, we assume that the input digraph of
the $\mathcal F$-free orientation problem is given with vertex set $[n]$
for $n = |V(D)|$. Clearly, this assumption is done without loss of
generality.

\begin{corollary}\label{cor:3-vertices-or-com}
    Let $\mathcal F$ be a finite set of finite tournaments. If the smallest
    tournament in $\mathcal F$ has exactly three vertices, then one of the following
    holds.
    \begin{enumerate}
        \item $\mathcal F$ contains both $T_3$ and $\overrightarrow{C_3}$.
        In this case, a digraph $D$ admits an $\mathcal F$-free orientation completion
        if and only if $D$ contains no semicomplete digraph on $3$ vertices. 
        \item $\mathcal F$ contains a $T_3$ but not $\overrightarrow{C_3}$ . In this case, 
        the $\mathcal F$-free orientations of a digraph $D$ correspond to the solution 
        space of Sys$_3(D)$.
         \item $\mathcal F$ contains a $\overrightarrow{C_3}$ but not $T_3$. 
        In this case, the $\mathcal F$-free orientation completion problem is
        NP-complete.
    \end{enumerate}
    In cases 1 and 2, the $\mathcal F$-free orientation completion problem is in P.
\end{corollary}
\begin{proof}
The first case is immediate to see. To prove the second case, notice
that every tournament on \red{at least} $4$ vertices contains a transitive triangle and so, 
an orientation completion $D'$ of $D$ is $\mathcal F$-free if and only if
it is  $T_3$-free. Thus, the second case follows from the two paragraphs
preceding this corollary. Finally, the case when 
$\overrightarrow{C_3}\in \mathcal{F}$ but $T_3 \not\in \mathcal F$
follows by noticing that $P_3^{\bB_{\mathcal F}}$ is not preserved by minority:
the tuples $(1,1,0)$, $(0,1,1)$, and $(1,1,1)$ correspond to three permutations
of $T_3$ (i.e., three $\mathcal F$-free tournaments), while
$\minority((1,1,0),(0,1,1),(1,1,1)) = (0,1,0)$ which corresponds to a directed
triangle (i.e., a non-$\mathcal F$-free tournament). Thus, the $\mathcal F$-free
orientation completion problem in NP-complete by~\cref{thm:or-comp-classification,}.
\end{proof}

Notice that if $G$ is a graph with vertex set $[n]$, then Sys$_3(G)$ does not
contain equations of the form $x_{ij} = 1$.

\begin{corollary}\label{cor:3-vertices-or}
    Let $\mathcal F$ be a finite set of finite tournaments. If the smallest
    tournament in $\mathcal F$ has exactly three vertices, then one of the following
    holds.
    \begin{enumerate}
        \item $\mathcal F$ contains both $T_3$ and $\overrightarrow{C_3}$.
        In this case, a graph $G$ admits an $\mathcal F$-free orientation 
        if and only if the $G$ is $K_3$-free. 
        \item $\mathcal F$ contains a $T_3$ but not $\overrightarrow{C_3}$ . 
        In this case, the $\mathcal F$-free orientations of a graph $G$
        correspond to the solution  space of Sys$_3(G)$.
         \item $\mathcal F$ contains $\overrightarrow{C_3}$ but not $T_3$. 
        In this case, every graph admits an $\mathcal F$-free orientation.
    \end{enumerate}
    In each of these cases, the $\mathcal F$-free orientation completion problem is in P.
\end{corollary}
\begin{proof}
    Cases 1 and 3 are trivial, and case 2 follows as a particular instance
    of the second case in \cref{cor:3-vertices-or-com}.
\end{proof}

It is also possible to have an ad-hoc reduction from an NP-complete
problem to the $\overrightarrow{C_3}$-orientation completion problem. 
The most natural problem in this scenario is not-all-equal $3$-SAT:
the input is a $3$-SAT instance; and a solution is where at least one
variable per clause is false, and one is true. As a sanity check, one
can easily verify that the gadget $(D,x_0,x_1,y_0,y_1,z_0,z_1)$ in
\cref{fig:gadget-NAE} yields a reduction from not-all-equal $3$-SAT
to the $\overrightarrow{C_3}$-free orientation. To see this,
simply notice that $(x_0,x_1)$ and $(a,b)$ force each other, and
symmetrically, $(y_0, y_1)$ and $(b,c)$ force each other, and  $(z_0,z_1)$
and $(c,a)$ force each other.
Since the triangle $abc$ must be oriented transitively, it must be the case
that at least one of the edge $x_0x_1$, $y_0y_1$, $z_0z_1$ is oriented from 
$0$ to $1$, but not the three of them. Moreover, any transitive orientation
of $abc$ extends to a $\overrightarrow{C_3}$-free orientation of $D$. With
 these arguments, one can notice that \red{given an instance} $\phi$ of not-all-equal
$3$-SAT, there is a digraph $D$ such that $D$ admits a $\overrightarrow{C_3}$-free
orientation completion if and only if $\phi$ is a yes instance
to not-all-equal $3$-\red{SAT} (and $D$ can be constructed in polynomial time
with respect to $D$). Actually, the reader familiar with primitive positive interpretations 
can notice that the gadget $(D,x_0,x_1,y_0,y_1,z_0,z_1)$ can be translated
into a primitive positive interpretation of $(\{0,1\},\{0,1\}^3\setminus\{(0,0,0),(1,1,1)\})$ 
in $D_\mathcal F$ --- which is guaranteed to exist by \cref{thm:or-comp-classification}
and the well-known fact that $K_3$ primitively positively interprets any finite structure~\cite{Book}, and that 
primitive positive interpretations compose. 

\begin{figure}[ht!]
\centering
\begin{tikzpicture}

  \begin{scope}[xshift = 0cm]
    \node [vertex, label = -27:{$a$}] (a) at (150:1) {};
    \node [vertex, label = 200:{$b$}] (b) at (30:1) {};
    \node [vertex, label = above:{$c$}] (c) at (-90:1) {};

    \node [vertex, label = above:{$x_0$}] (x0) at (-cos{30}, 3.7) {};
    \node [vertex, label = above:{$x_1$}] (x1) at (cos{30},3.7) {};
    \node [vertex] (1) at (-cos{30}-2*sin{120},2.1) {};
    \node [vertex] (2) at (-cos{30},2.1) {};
    \node [vertex] (3) at (cos{30},2.1) {};
    \node [vertex] (4) at (cos{30}+2*sin{120},2.1) {};
      
    \foreach \from/\to in {a/b, b/c, c/a, x0/x1, x0/3, x1/1, x0/4,
    1/2, 3/4, 2/a, b/3}     
    \draw [edge] (\from) to (\to);

    \foreach \from/\to in {1/x0, x1/2, x0/3, 4/x1, a/1, 2/b, b/4, 3/a}     
    \draw [arc] (\from) to (\to);

    \node [vertex, label = left:{$z_0$}] (z0) at ({300-76.8265}:3.8) {};
    \node [vertex, label = left:{$z_1$}] (z1) at (196.8264:3.8) {};
    \node [vertex] (1) at ({300-38.9483}:3.34066) {};
    \node [vertex] (2) at ({300-67.5891}:2.27156) {};
    \node [vertex] (3) at (187.5891:2.27156) {};
    \node [vertex] (4) at (158.9483:3.34066) {};
      
    \foreach \from/\to in {z0/z1, z0/3, z1/1, z0/4, 1/2, 3/4, 2/c, a/3}     
    \draw [edge] (\from) to (\to);

    \foreach \from/\to in {1/z0, z1/2, z0/3, 4/z1, c/1, 2/a, a/4, 3/c}     
    \draw [arc] (\from) to (\to);
    
    \node [vertex, label = right:{$y_0$}] (y0) at ({60-76.8265}:3.8) {};
    \node [vertex, label = right:{$y_1$}] (y1) at (316.8264:3.8) {};
    \node [vertex] (1) at ({60-38.9483}:3.34066) {};
    \node [vertex] (2) at ({60-67.5891}:2.27156) {};
    \node [vertex] (3) at (307.5891:2.27156) {};
    \node [vertex] (4) at (278.9483:3.34066) {};
      
    \foreach \from/\to in {y0/y1, y0/3, y1/1, y0/4, 1/2, 3/4, 2/b, c/3}     
    \draw [edge] (\from) to (\to);

    \foreach \from/\to in {1/y0, y1/2, y0/3, 4/y1, b/1, 2/c, c/4, 3/b}     
    \draw [arc] (\from) to (\to);

    \node (L1) at (0,-4) {$(D,x_0,x_1,y_0,y_1,z_0,z_1)$};
    
  \end{scope}
  
\end{tikzpicture}
\caption{A gadget for reducing not-all-equal $3$-\red{SAT} to 
the $\overrightarrow{C_3}$-free orientation completion. Equivalently, 
the interpreting formula for the primitive positive interpretation of
$(\{0,1\},\{0,1\}^3\setminus\{(0,0,0),(1,1,1)\})$ in  $(D_\mathcal F, U)$
for $\mathcal F = \{\overrightarrow{C_3}\}$.}
\label{fig:gadget-NAE}
\end{figure}
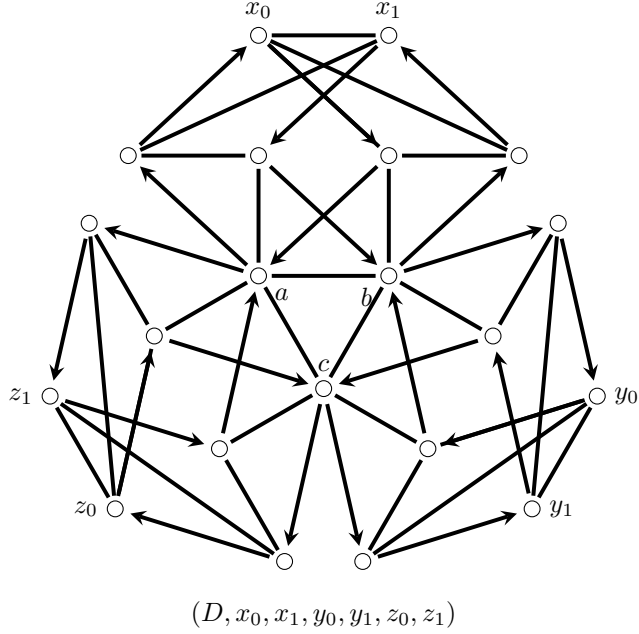

\subsection{Tournaments on four vertices}
In the previous subsection we considered all cases where  $\mathcal F$ contains
at least one tournament on at most $3$ vertices. In this subsection, we apply
our main results to classify the complexity of the $\mathcal F$-free orientation
and orientation completion problem if all tournaments in $\mathcal F$ have exactly
$4$ vertices.
Similarly as in the previous subsection, we introduce three families of systems
of linear equations over $\mathbb Z_2$, whose solution space will encode
$\mathcal F$-free orientation completions for certain sets $\mathcal F$.

Recall that in \cref{fig:four-vertices} we depicted all four tournaments
on $4$ vertices, up to isomorphism. Notice that if $T_4$ is described by
the linear ordering of $[4]$, then the cycle $1342$ has $3$ forward edges
and $1$ backward edge, i.e., $x_{12} + x_{23} + x_{34} + x_{41} = 1$ 
according to the coding described in the previous subsection. Similarly, 
it is not hard to find a hamiltonian oriented cycle $ijkl$ in $TC_4$  such that 
$x_{ij} + x_{jk} + x_{kl} + x_{li} = 1$. On the contrary, it is evident
that every hamiltonian cycle $ijkl$ of $C_3^+$ and of $C_3^-$ have an even number
of forward edges, i.e., $x_{ij} + x_{jk} + x_{kl} + x_{li} = 0$. Thus, similar
as the set up of Sys$_3(D)$, there is a system of linear equations Sys$_4(D)$
whose solution space correspond to the $\{T_4,TC_4\}$-free orientation completions
of a digraph $D$. 

It is even simpler to verify that each vertex in $C_3^+$ has an even number
of outneighbours, while for $T_4$, $TC_4$, and $C_3^-$, there is at least
one vertex with an odd number of outneighours. Again, this leads to a system
of linear equations Sys$_+(D)$ with variables $x_{ij}$ where $ij\in U(D)$,
such that the solution space (over $\mathbb Z_2$) of Sys$_+(D)$ corresponds to
the $\{T_4,TC_4,C_3^-\}$-free orientation completions of $D$. Moreover, notice
that $C_3^-$ is the flip of $C_3^+$, thus an oriented graph $D'$ 
is $\{T_4,TC_4,C_3^-\}$-free if and only $-D$  is $\{T_4,TC_4,C_3^+\}$-free.
This means that if $x$ is a solution to Sys$_+(D)$ and we interpret
$x_{ij} = 1$ as an edge $(j,i)$ (instead of $(i,j)$), we see that
the solution space of Sys$_+(D)$ also encondes the $\{T_4,TC_4,C_3^+\}$-free
orientation completions of $D$.

\begin{corollary}\label{cor:4-vertices-or-com}
    If $\mathcal F$ is a non-empty set of tournaments on $4$ vertices, then one of the following
    holds.
    \begin{enumerate}
        \item $\mathcal F$ contains all tournaments on $4$ vertices, up to isomorphism. 
        In this case, a digraph $D$ admits an $\mathcal F$-free orientation completion
        if and only if it does not contain a semicomplete digraph on $4$ vertices.
        \item  $\mathcal F$ contains both $T_4$ and $TC_4$. In this case, 
        the $\mathcal F$-free orientation completions of a digraph $D$
        correspond to the solution space of Sys$_4(D)$ or of Sys$_+(D)$. 
        \item $\mathcal F$ contains at most one of $T_4$ or $TC_4$. In this case,
         the $\mathcal F$-free orientation completion is NP-complete. 
    \end{enumerate}
    In cases 1 and 2, the $\mathcal F$-free orientation completion problem is in P.
\end{corollary}
\begin{proof}
The first case is immediate, and the second one is argued in the paragraph preceding this statement. 
To prove item 3, first notice that the case when $T_4\in \mathcal F$ and $TC_4\not\in \mathcal F$,
follows from \cref{ex:T4-TC4}. Otherwise, suppose that $T_4\not\in \mathcal F$, and so, $T_4$
is $\mathcal F$-free. Since $\mathcal F$ contains at least one tournament on $4$ vertices, and 
$T_4$ is $\mathcal F$-free,  it follows from \cref{lem:Tn-collapse},  that $P_n$ is not preserved
by the Boolean minority operation. The hardness of the $\mathcal F$-free orientation completion
problem now follows from \cref{thm:or-comp-classification}.
\end{proof}

\begin{corollary}\label{cor:4-vertices-or}
    If $\mathcal F$ is a set of tournaments on $4$ vertices, then one of the following
    holds.
    \begin{enumerate}
        \item $\mathcal F$ contains all tournaments on $4$ vertices, up to isomorphism. 
        In this case, a graph $G$ admits an $\mathcal F$-free orientation 
        if and only if $G$ is $K_4$-free.
         \item  $\mathcal F$ contains both $T_4$ and $TC_4$. In this case, 
        the $\mathcal F$-free orientations of a graph $G$ correspond to the solution
        space of Sys$_4(G)$ or of Sys$_+(G)$.
        \item $\mathcal F$ does not contain $T_4$. In this case, 
        every graph admits an $\mathcal F$-free orientation.
         \item $\mathcal F$ contains $T_4$ but not $TC_4$.
        In this case, the $\mathcal F$-free orientation problem is NP-complete.
    \end{enumerate}
    In cases 1-3, the $\mathcal F$-free orientation completion problem is in P.
\end{corollary}
\begin{proof}
Cases 1 and 2 follow as particular instances of cases 1 and 2 from \cref{cor:4-vertices-or-com}. 
Case 3 is trivial, and the fourth one follows from~\cref{ex:T4-TC4}.
\end{proof}

\subsection{Transitive Tournaments}
We conclude our series of examples by considering the cases where  
$\mathcal F$ does not contain a transitive tournament, and where 
$\mathcal F$ only contains transitive tournaments. 

\begin{proposition}\label{thm:or-comp-transitive}
    Let $\mathcal F$ be a non-empty finite set of finite tournaments. If $\mathcal F$
    does not contain any transitive tournament, then the following statements hold.
    \begin{enumerate}
        \item The $\mathcal F$-free orientation completion problem is NP-complete. 
        Moreover, this problem remains NP-hard even when the input is restricted
        to digraphs with no semicomplete subdigraph with $m_\mathcal F+1$ vertices.
        \item The $\mathcal F$-free orientation problem is trivial and in P.
    \end{enumerate}
\end{proposition}
\begin{proof}
    We only prove the first statement as the second one is evident. Let  $\mathcal F'$ 
    be the set obtained from $\mathcal F$ by adding all tournament\red{s} on $m_\mathcal F +1$ vertices.
    It is not hard to notice that the $\mathcal F$-free orientation completion problem restricted
    to digraphs \red{with} no semicomplete digraph on $m_\mathcal F +1$ vertices is polynomial time
    equivalent to the $\mathcal F'$-free orientation completion problem: on the one hand, 
    if $D$ is a digraph with no semicomplete graph on $m_\mathcal F+1$ vertices, then 
    an orientation completion $D'$ of $D$ is $\mathcal F$-free if and only if $D'$ is 
    $\mathcal F'$-free; on the other \red{hand}, if $D$ is an input to the $\mathcal F'$-free
    orientation completion, then one can first verify (in polynomial-time) if $D$ contains
    a semicomplete digraph on $m_\mathcal F + 1$-vertices (if it does, reject), and \red{then}
    solve the $\mathcal F$-free orientation completion problem. 
    Finally, the fact that the $\mathcal F'$-free orientation completion problem is NP-complete,
    follows from the observation that $T_{m_\mathcal F}$ is $\mathcal F'$-free, but  since
    $\mathcal F$ is non-empty, there must be at lest one tournament $T$ on $m_\mathcal F$ vertices
    that is \red{not} $\mathcal F$-free. Thus, $T$ is not $\mathcal F'$-free but $T_{m_\mathcal F}$ is
    $\mathcal F'$-free, hence, by \cref{lem:collapse} we conclude that $P_{m_\mathcal F'}$ cannot
    be preserved by the minority operation. The hardness of the $\mathcal F'$-free orientation
    completion problem now follows from \cref{thm:or-comp-classification}, and so, the
    $\mathcal F$-free orientation completion problem is NP-complete even when the input
    is restricted so digraphs with no semicomplete digraph on $m_\mathcal F+1$ vertices.
\end{proof}

We highlight that \cref{thm:or-comp-transitive} provides several instances
of sets $\mathcal{F}$ such that the $\mathcal{F}$-free orientation problem
is in $P$, while the $\mathcal{F}$-free orientation completion problem
is $\NP$-complete. It turns out that in ``almost'' any other case, the
orientation and orientation completion problems are equivalent. 

\begin{theorem}\label{thm:or-eq-or-comp}
    For a finite set of non-empty finite tournament $\mathcal F$ one of the following statements
    holds.
    \begin{enumerate}
        \item $\mathcal F$ contains a transitive tournament and at least one
        tournament with $n_\mathcal F$ vertices is $\mathcal F$-free. In this case, the
        $\mathcal F$-free orientation and the $\mathcal F$-free orientation completion problems
        are polynomial time equivalent.
        \item $\mathcal F$ contains a transitive tournament and all tournaments
        in $\mathcal F$ have at least $n_\mathcal F$ vertices. In this case, the
        $\mathcal F$-free orientation and the $\mathcal F$-free orientation completion problems
        are polynomial time equivalent.
        \item Otherwise, the $\mathcal F$-free orientation  completion problem is NP-complete, and the
        $\mathcal F$-free orientation problem is in P.
    \end{enumerate}
\end{theorem}
\begin{proof}
    In case 1, it follows from~\cref{thm:or-comp-CSP-no-Tk}, that the $\mathcal F$-free
    orientation completion problem an $\Csp(\bB_\mathcal F)$ are polynomial time equivalent. 
    And thus, since $\Csp(\bB_\mathcal F)$ and the $\mathcal F$-free orientation completion
    problem are polynomial-time equivalent, the claim follows. To prove the second statement, 
    notice that if $\mathcal F$ does not contain all tournaments on $n_\mathcal F$ vertices, 
    then we are in case 1. Now, if $\mathcal F$ contains all tournaments on $n_\mathcal F$
    vertices, then the $\mathcal F$-free orientation corresponds to finding $n_\mathcal F$
    complete graphs in an input graph $G$, and the $\mathcal F$-free orientation completion
    problem corresponds to finding semicomplete graphs on $n_\mathcal F$ vertices in the input
    digraph $D$. Thus, both problems are in P, and we conclude that the second statement holds. 
    
    Now we prove the third statement. If $\mathcal F$ does not contain a transitive tournament, 
    the claim follows from \cref{thm:or-comp-transitive}. Otherwise, $\mathcal F$ contains
    a transitive tournament, there is no $\mathcal F$-free tournament on $n_\mathcal F$ vertices,
    and $\mathcal F$ contains some tournament on less than $n_\mathcal F$ vertices. On the one hand, 
    this means that the $\mathcal F$-free orientation problem reduces to determining if the input
    graph is $K_{n_\mathcal F}$-free, and thus it is in P. On the other \red{hand}, if $\mathcal F'$
    is obtained from $\mathcal F$ by removing all tournament on $n_\mathcal F$ vertices, 
    then the $\mathcal F'$-free orientation completion problem in NP-complete by \cref{thm:or-comp-transitive}.
    Moreover, it \red{remains} NP-complete when restricted to input digraphs with no semicomplete
    digraph on $m_{\red{\mathcal F'}}+1$ vertices. This problem is clearly equivalent to the restriction of the
    $\mathcal F$-free orientation completion problem to digraphs with no semicomplete
    digraph on $m_{\red{\mathcal F'}}+1$ vertices. Thus, the (general) $\mathcal F$-free orientation
    completion problem must be NP-complete as well. 
\end{proof}

On the opposite side of the cases \red{considered} in \cref{thm:or-comp-transitive} is the case
where $\mathcal F$ only contains transitive tournaments. In this case, by \cref{thm:or-eq-or-comp},
the $\mathcal F$-free orientation completion and the $\mathcal F$-free orientation problems
are polynomial-time equivalent. Also, notice that these cases boil down to the case when
$\mathcal F = \{T_k\}$ for some integer $k$ (larger forbidden transitive tournaments are redundant).

\begin{theorem}\label{thm:Tk}
    The following statements hold for a positive integer $k\ge 4$.
    \begin{enumerate}
        \item The $T_k$-free orientation completion problem is NP-complete. 
        This problem remains NP-complete when the input is restricted to digraphs with no semicomplete
        subdigraph with $k+1$ vertices.
        \item The $T_k$-free orientation problem is NP-complete. 
        This problem remains NP-complete when the input is restricted to $K_{k+1}$-free graphs.
    \end{enumerate}
\end{theorem}
\begin{proof}
    Clearly, it suffices to prove the second part of statements 1 and 2. 
    Let $\mathcal F_k$ be the set of tournaments that contains $T_k$ and all tournaments
    on $k+1$ vertices. With similar arguments as in the proof of \cref{thm:or-comp-transitive},
    one can notice that the $\mathcal F_k$-free orientation (resp.\ completion) problem is
    polynomial-time equivalent to the $T_k$-free orientation (resp.\ completion) problem when the
    input is restricted to  $K_{k+1}$-free graphs (resp.\ digraphs with no subdigraph with
    $k+1$ vertices). Moreover, by the second statement of~\cref{thm:or-eq-or-comp}, the 
    $\mathcal F_k$-free orientation and the $\mathcal F_k$-free orientation completion problems
    are polynomial-time equivalent. Therefore, to prove the whole theorem, if suffices to
    prove the $\mathcal F_k$-free orientation completion problem is NP-complete for $k\ge 4$.
    To do so, notice that if $D$ is a digraph and $D'$ is obtained from $D$ by adding a universal
    sink, then $D$ can be completed to a  $\mathcal F_k$-free oriented graph if and only if
    $D'$ can be completed to a $\mathcal F_{k+1}$-free oriented graph. Thus, the $\mathcal F_k$-free
    orientation completion problem reduces in polynomial-time to the  $\mathcal F_{k+1}$-free orientation
    completion problem. Hence, we conclude the proof by showing that the $\mathcal F_4$-free orientation
    completion problem is NP-complete. Clearly, $T_4\in \mathcal F_4$ and $TC_4$ is $\mathcal F_4$-free
    so, it follows from~\cref{ex:T4-TC4} that the $\mathcal F_4$-free orientation completion
    problem is NP-complete. Both statements now follow.
\end{proof}


\section{Conclusion and Outlook}


From a structural perspective, a family of graph obstructions to the class of graphs
that admit a $T_3$-free orientation was described in \cite{GuzmanProForbiddenPatterns}.
In light of the hardness of the $T_4$-free orientation problem (\cref{cor:4-vertices-or}),
it might be hard to extend such a description for the class of graphs that admit a $T_4$-free
orientation, but it could be interesting to understand the structure
of graphs that admit a  $\{T_4,TC_4\}$-free orientation, and of those that admit a
$\{T_4,TC_4, C_3^+\}$-free orientation.

From a computational complexity point of view, a first natural extension of this work would be
to classify the complexity of the $\mathcal F$-free orientation (completion) problem
if $\mathcal F$ is any finite set of oriented graphs. In general, this might
not be equivalent to a (possibly infinite) CSP, but if $\mathcal F$ consists of connected oriented graphs, and it is closed
under homomorphisms, then the class of graphs that admit an $\mathcal F$-free
orientation corresponds to the CSP of some (possibly infinite) graph $G$
(see e.g.,~\cite{Book}). 
Such a restriction on the forbidden set $\mathcal F$ is also a particular instance
of the larger class of problems that can be expressed in the logic MMSNP$_2$. 
Some of the techniques we used to classify the computational complexity of the ${\mathcal F}$-free orientation problem might be useful to prove a complexity dichotomy 
for the class of problems  that can be expressed in the logic 
MMSNP$_2$. One would have to overcome the following obstacles.
\begin{itemize}
    \item For problems expressible in MMSNP$_2$, the finite structures we work with instead of $\bB_{\mathcal F}$ and $\bC_{\mathcal F}$ will in general have more than two elements, which means that we cannot use lemmata that explicitly rely on the Schaefer's cases, such as Lemma~\ref{lem:collapse}. 
    \item For the ${\mathcal F}$-free orientation problem, the concept of \emph{force} that we introduced in Section~\ref{sect:dicho-oc} is particularly pleasant, since flipping arguments corresponds to Boolean complementation. 
    The combinatorics of forcing will be more involved in the general case. 
    \item For the structures needed to formulate problems MMSNP$_2$ as CSPs, there is no known generalisation of the result of Kompatscher and Agarval~\cite{AgarwalKompatscher}, which was crucial in our proof. 
\end{itemize}
However, our hope is that a combination of ideas from the present paper with more recent results in the theory of constraint satisfaction, e.g., from~\cite{MottetPinskerSmooth,MottetPinskerCores,BodirskyBodorUIP,MMSNP-Journal} can eventually lead to a proof of a complexity dichotomy for all of MMSNP$_2$. 
\red{After the submission of the present article, 
our result has been reproved with the theory of smooth approximations by Feller and Pinsker~\cite{FellerP} and 
generalised to larger fragments of MMSNP$_2$ by Bitter and Mottet~\cite{BitterM}.}



\subsection*{Acknoweledgements}
\red{We thank the anonymous referees for their valuable suggestions and Roman Feller for pointing out a mistake in a preprint of this article.}  

\bibliographystyle{abbrv}
\bibliography{global.bib}

\end{document}